\documentclass[11pt]{article}

\usepackage{amsfonts,amssymb,amsmath,amsthm,epsfig,euscript,epstopdf}
\usepackage[noadjust]{cite}
\usepackage{filecontents}
\usepackage{hyperref}

\newtheorem{theorem}{Theorem}
\newtheorem{lemma}[theorem]{Lemma}

\long\def\symbolfootnote[#1]#2{\begingroup
\def\thefootnote{\fnsymbol{footnote}}\footnote[#1]{#2}\endgroup}

\newcommand{\Pos}{\mathbb{P}}

\newcommand{\rank}{\mathrm{rank}}

\newcommand{\cref}[1]{Corollary \ref{corollary:#1}}

\newcommand{\fig}[2]{\begin{figure}[ht]
\centerline{\scalebox{.66}{\epsfig{file=#1.eps}}}
\caption{#2}
\label{fig:#1}
\end{figure}}

\title{ $Q$-analogues of the Fibo-Stirling numbers}

\author{
Quang T. Bach \\
\small Department of Mathematics\\[-0.8ex]
\small University of California, San Diego\\[-0.8ex]
\small La Jolla, CA 92093-0112. USA\\[-0.8ex]
\small \texttt{qtbach@ucsd.edu}
\and
Roshil Paudyal\\
\small{Department of Mathematics} \\
\small{Howard University}\\
\small \texttt{roshil.paudyal@bison.howard.edu}
\and
Jeffrey B. Remmel \\
\small Department of Mathematics\\[-0.8ex]
\small University of California, San Diego\\[-0.8ex]
\small La Jolla, CA 92093-0112. USA\\[-0.8ex]
\small \texttt{remmel@math.ucsd.edu}
\and
}

\date{\small Submitted: Date 1;  Accepted: Date 2;
 Published: Date 3.\\
\small MR Subject Classifications: 05A15, 05E05 \\
keywords: Fibonacci numbers, Stirling numbers, Lah numbers}

\begin{document}

\maketitle

\begin{abstract}

Let $F_n$ denote the $n^{th}$ Fibonacci number relative to the initial 
conditions $F_0=0$ and $F_1=1$. 
In \cite{BPR}, we introduced Fibonacci analogues of the Stirling 
numbers called Fibo-Stirling numbers 
of the first and second kind. These numbers serve as the connection coefficients 
between the Fibo-falling factorial basis 
$\{(x)_{\downarrow_{F,n}}:n \geq 0\}$ and the Fibo-rising factorial 
basis $\{(x)_{\uparrow_{F,n}}:n \geq 0\}$ which are defined by  
$(x)_{\downarrow_{F,0}} = (x)_{\uparrow_{F,0}} = 1$ and for 
$k \geq 1$, $(x)_{\downarrow_{F,k}} = x(x-F_1) \cdots (x-F_{k-1})$ and 
$(x)_{\uparrow_{F,k}} = x(x+F_1) \cdots (x+F_{k-1})$. 
We gave a general rook theory model which allowed us to give combinatorial 
interpretations of the Fibo-Stirling numbers of the first and second kind.  

There are two natural 
$q$-analogues of the falling and rising Fibo-factorial basis. That is, let 
$[x]_q = \frac{q^x-1}{q-1}$. Then we let 
$[x]_{\downarrow_{q,F,0}} = \overline{[x]}_{\downarrow_{q,F,0}} = 
[x]_{\uparrow_{q,F,0}} = \overline{[x]}_{\uparrow_{q,F,0}}=1$ and, for 
$k > 0$, we let \\
$[x]_{\downarrow_{q,F,k}} = [x]_q [x-F_1]_q \cdots [x-F_{k-1}]_q$, 
$\overline{[x]}_{\downarrow_{q,F,k}}= [x]_q ([x]_q-[F_1]_q) \cdots ([x]_q-[F_{k-1}]_q)$, \\
$[x]_{\uparrow_{q,F,k}}=  [x]_q [x+F_1]_q \cdots [x+F_{k-1}]_q$, and   
$\overline{[x]}_{\uparrow_{q,F,k}}= [x]_q ([x]_q+[F_1]_q) \cdots 
([x]_q+[F_{k-1}]_q)$. 

In this paper, we show we can modify the  
rook theory model of \cite{BPR} to give combinatorial interpretations 
for the two different types $q$-analogues of the Fibo-Stirling numbers which 
arise as the connection coefficients between the two different $q$-analogues 
of the Fibonacci falling and rising factorial bases. 
\end{abstract}

\section{Introduction}

Let $\mathbb{Q}$ denote the rational numbers and 
$\mathbb{Q}[x]$ denote the ring of polynomials over $\mathbb{Q}$.
Many classical combinatorial sequences can be defined as 
connection coefficients between various basis of 
the polynomial ring $\mathbb{Q}[x]$. 
There are three very natural bases for $\mathbb{Q}[x]$. 
The usual power basis $\{x^n: n\geq 0\}$, 
the falling factorial basis  $\{(x)_{\downarrow_n}: n\geq 0\}$, 
and the rising factorial basis  $\{(x)_{\uparrow_n}: n\geq 0\}$. 
Here we let $(x)_{\downarrow_0} = (x)_{\uparrow_0} = 1$ and for 
$k \geq 1$, $(x)_{\downarrow_k} = x(x-1) \cdots (x-k+1)$ and 
$(x)_{\uparrow_k} = x(x+1) \cdots (x+k-1)$.
Then the Stirling numbers of the first kind $s_{n,k}$, 
the Stirling numbers of the second kind $S_{n,k}$ and 
the Lah numbers $L_{n,k}$ are defined by specifying 
that for all $n \geq 0$, 
\begin{equation*}
(x)_{\downarrow_n} = \sum_{k=1}^n s_{n,k} \ x^k, \ \ 
x^n = \sum_{k=1}^n S_{n,k} \ (x)_{\downarrow_k}, \ \mbox{and} \ \ 
(x)_{\uparrow_n} = \sum_{k=1}^n L_{n,k} \ (x)_{\downarrow_k}.
\end{equation*}
The signless Stirling numbers of the first kind 
are defined by setting $c_{n,k} = (-1)^{n-k} s_{n,k}$. 
Then it is well known that $c_{n,k}$, $S_{n,k}$, and $L_{n,k}$ 
can also be defined by the recursions that 
$c_{0,0} = S_{0,0} = L_{0,0} = 1$, 
$c_{n,k} = S_{n,k} = L_{n,k} = 0$ if either $n < k$ or 
$k < 0$, and 
\begin{equation*}
c_{n+1,k} = c_{n,k-1}+ n c_{n,k}, \ \ 
S_{n+1,k} =  S_{n,k-1}+kS_{n,k}, \ \mbox{and} \ \  
L_{n+1,k} = L_{n,k-1} +(n+k) L_{n,k}
\end{equation*}
for all $n,k \geq 0$. 
There are well known combinatorial interpretations of 
these connection coefficients. That is, 
$S_{n,k}$ is the number of set partitions of $[n] = \{1, \ldots, n\}$ 
into $k$ parts, $c_{n,k}$ is the number of permutations 
in the symmetric group $S_n$ with $k$ cycles, and 
$L_{n,k}$ is the number of ways to place $n$ labeled balls into $k$ unlabeled 
tubes with at least one ball in each tube.

In \cite{BPR}, we introduced Fibonacci analogues of 
the number $s_{n,k}$, $S_{n,k}$, and $L_{n,k}$.  
We started with the tiling model of the $F_n$ of \cite{SS}. 
That is, let $\mathcal{FT}_n$ denote the set of tilings  
a column of height $n$ with tiles of height  
1 or 2 such that bottom most tile is of height 1. 
For example, possible tiling configurations 
for $\mathcal{FT}_i$ for $i \leq 4$ are shown in 

\fig{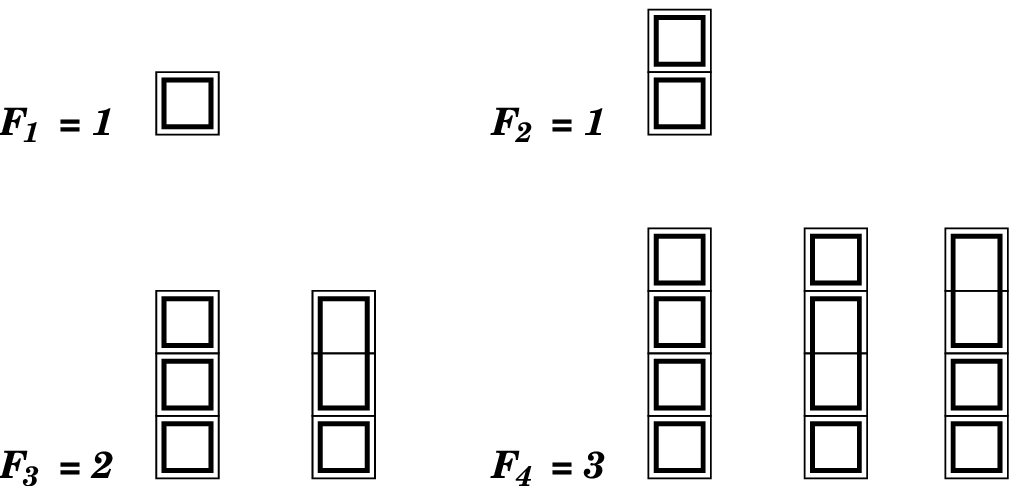}{The tilings counted by $F_i$ for $1 \leq i \leq 4$.}

For each tiling $T \in \mathcal{FT}_n$, we let 
$\mathrm{one}(T)$ is the number of tiles of 
height 1 in $T$ and $\mathrm{two}(T)$ is the number of tiles of 
height 2 in $T$ and define 
\begin{equation*}
F_n(p,q) = \sum_{T \in \mathcal{FT}_n} q^{\mathrm{one}(T)}p^{\mathrm{two}(T)}.
\end{equation*} 
It is easy to see that 
 $F_1(p,q) =q$, $F_2(p,q) =q^2$ and $F_n(p,q) =q F_{n-1}(p,q)+ 
pF_{n-2}(p,q)$ for $n \geq 2$ so that $F_n(1,1) =F_n$.
We then 
defined the $p,q$-Fibo-falling 
factorial basis 
$\{(x)_{\downarrow_{F,p,q,n}}:n \geq 0\}$ and the $p,q$-Fibo-rising factorial 
basis $\{(x)_{\uparrow_{F,p,q,n}}:n \geq 0\}$ by setting 
$(x)_{\downarrow_{F,p,q,0}} = (x)_{\uparrow_{F,p,q,0}} = 1$ and setting
\begin{eqnarray*}
 (x)_{\downarrow_{F,p,q,k}} &=& x(x-F_1(p,q)) \cdots (x-F_{k-1}(p,q)) \ \mbox{and} \\ 
(x)_{\uparrow_{F,p,q,k}} &=& x(x+F_1(p,q)) \cdots (x+F_{k-1}(p,q))
\end{eqnarray*}
for $k \geq 1$. 

Our idea to define $p,q$-Fibonacci analogues of 
the Stirling numbers of the first kind, $\mathbf{sf}_{n,k}(p,q)$, 
the Stirling numbers of the second kind, $\mathbf{Sf}_{n,k}(p,q)$, 
and the Lah numbers, $\mathbf{Lf}_{n,k}(p,q)$, is to 
define them   
to be the connection coefficients between the usual power basis 
$\{x^n: n \geq 0\}$ and the $p,q$-Fibo-rising factorial and 
$p,q$-Fibo-falling factorial bases. 
That is, we define $\mathbf{sf}_{n,k}(p,q)$, $\mathbf{Sf}_{n,k}(p,q)$, 
and $\mathbf{Lf}_{n,k}(p,q)$ by the equations  
\begin{equation}\label{pqsfdef}
(x)_{\downarrow_{F,p,q,n}} = \sum_{k=1}^n \mathbf{sf}_{n,k}(p,q) \ x^k,
\end{equation}
\begin{equation}\label{pqSfdef}
x^n = \sum_{k=1}^n \mathbf{Sf}_{n,k}(p,q) \ (x)_{\downarrow_{F,p,q,k}}, \ \mbox{and} 
\end{equation}
\begin{equation}\label{pqLfdef}
(x)_{\uparrow_{F,p,q,n}} = \sum_{k=1}^n \mathbf{Lf}_{n,k}(p,q) \ (x)_{\downarrow_{F,p,q,k}}
\end{equation}
for all $n \geq 0$. 

It is easy to see that these equations imply simple recursions for 
the connection coefficients $\mathbf{sf}_{n,k}(p,q)$s, $\mathbf{Sf}_{n,k}(p,q)$s, and $\mathbf{Lf}_{n,k}(p,q)$s. 
That is, $\mathbf{sf}_{n,k}(p,q)$s, $\mathbf{Sf}_{n,k}(p,q)$s, and $\mathbf{Lf}_{n,k}(p,q)$s can be defined by the following recursions 
\begin{eqnarray*}
\mathbf{sf}_{n+1,k}(p,q) &=& \mathbf{sf}_{n,k-1}(p,q)- F_n(p,q) \mathbf{sf}_{n,k}(p,q), \\
\mathbf{Sf}_{n+1,k}(p,q) &=& \mathbf{Sf}_{n,k-1}(p,q)+ F_k(p,q) \mathbf{Sf}_{n,k}(p,q), \ \mbox{and} \\
 \mathbf{Lf}_{n+1,k}(p,q) &=& \mathbf{Lf}_{n,k-1}(p,q)+ (F_k(p,q) +F_n(p,q))\mathbf{Lf}_{n,k}(p,q)
\end{eqnarray*}
plus the boundary 
conditions 
$$\mathbf{sf}_{0,0}(p,q)=\mathbf{Sf}_{0,0}(p,q)=\mathbf{Lf}_{0,0}(p,q)=1$$ 
and 
$$\mathbf{sf}_{n,k}(p,q) =\mathbf{Sf}_{n,k}(p,q) =\mathbf{Lf}_{n,k}(p,q) =0$$
 if $k > n$ or $k < 0$. 
If we define $\mathbf{cf}_{n,k}(p,q):= (-1)^{n-k}\mathbf{sf}_{n,k}(p,q)$, then 
$\mathbf{cf}_{n,k}(p,q)$s can be defined by the recursions   
\begin{equation}\label{pqcfrec}
\mathbf{cf}_{n+1,k}(p,q) = \mathbf{cf}_{n,k-1}(p,q)+F_n(p,q) \mathbf{cf}_{n,k}(p,q) 
\end{equation}  plus the boundary 
conditions $\mathbf{cf}_{0,0}(p,q)=1$ and $\mathbf{cf}_{n,k}(p,q) =0$ if $k > n$ or $k < 0$. It also follows that 
\begin{equation}\label{pqcfdef}
(x)_{\uparrow_{F,p,q,n}} = \sum_{k=1}^n \mathbf{cf}_{n,k}(p,q) \ x^k.
\end{equation}
In \cite{BPR}, we developed a new rook theory model to give a 
combinatorial interpretation of  
the $\mathbf{cf}_{n,k}(p,q)$s and the $\mathbf{Sf}_{n,k}(p,q)$s and to give 
combinatorial  proofs of their basic properties. 
This new rook theory model is 
a modification of the rook theory model for $S_{n,k}$ 
and $c_{n,k}$ except that we replace rooks by Fibonacci 
tilings.

 The main goal of this paper 
is to show how that model can be modified to 
give combinatorial interpretations to two new  
$q$-analogues of the $\mathbf{cf}_{n,k}(1,1)$s and the 
$\mathbf{Sf}_{n,k}(1,1)$s.
Let $[0]_q =1$ and $[x]_q = \frac{1-q^x}{1-q}$.  When 
$n$ is a positive integer, then $[n]_q = 1+ q+ \cdots +q^{n-1}$ is 
the usual $q$-analogue of $n$.  Then there are two natural 
analogues of the falling and rising Fibo-factorial basis. First we let  
$[x]_{\downarrow_{q,F,0}} = \overline{[x]}_{\downarrow_{q,F,0}} = 
[x]_{\uparrow_{q,F,0}} = \overline{[x]}_{\uparrow_{q,F,0}}=1$.  For $k > 0$, 
we let 
$k > 0$, 
\begin{eqnarray*}
\ [x]_{\downarrow_{q,F,k}} &=& [x]_q [x-F_1]_q \cdots [x-F_{k-1}]_q, \\
\ \overline{[x]}_{\downarrow_{q,F,k}}&=& [x]_q ([x]_q-[F_1]_q) \cdots ([x]_q-[F_{k-1}]_q), \\
\ [x]_{\uparrow_{q,F,k}}&=& [x]_q [x+F_1]_q \cdots [x+F_{k-1}]_q, \ \mbox{and} \\
\ \overline{[x]}_{\uparrow_{q,F,k}}&=& [x]_q ([x]_q+[F_1]_q) \cdots 
([x]_q+[F_{k-1}]_q).
\end{eqnarray*}
Then we define $\mathbf{cF}_{n,k}(q)$ and $\overline{\mathbf{cF}}_{n,k}(q)$ 
by the equations 
\begin{equation}\label{cFdef}
[x]_{\uparrow_{q,F,n}}= \sum_{k=1}^n \mathbf{cF}_{n,k}(q) [x]_q^k
\end{equation}
and 
\begin{equation}\label{ovcFdef}
\overline{[x]}_{\uparrow_{q,F,n}}= \sum_{k=1}^n 
\overline{\mathbf{cF}}_{n,k}(q)[x]_q^k.
\end{equation}
Similarly, we define $\mathbf{SF}_{n,k}(q)$ and 
$\overline{\mathbf{SF}}_{n,k}(q)$ 
by the equations 
\begin{equation}\label{SFdef}
[x]_q^n= \sum_{k=1}^n \mathbf{SF}_{n,k}(q) [x]_{\downarrow_{q,F,k}}
\end{equation}
and 
\begin{equation}\label{ovSFdef}
[x]_q^n= \sum_{k=1}^n 
\overline{\mathbf{SF}}_{n,k}(q) \overline{[x]}_{\downarrow_{q,F,k}}
\end{equation}

One can easily find recursions for these polynomials. 
For example, 
\begin{eqnarray*}
[x]_q^{n+1}&=& \sum_{k=1}^{n+1} \mathbf{SF}_{n+1,k}(q) 
[x]_{\downarrow_{q,F,k}} = \sum_{k=1}^n \mathbf{SF}_{n,k}(q) [x]_{\downarrow_{q,F,k}}[x]_q \\
&=& \sum_{k=1}^n \mathbf{SF}_{n,k}(q) [x]_{\downarrow_{q,F,k}}([F_k]_q 
+ q^{F_k}[x-F_k]_q)\\
&=& \sum_{k=1}^n [F_k]_q \mathbf{SF}_{n,k}(q) [x]_{\downarrow_{q,F,k}}
+\sum_{k=1}^n q^{F_k}\mathbf{SF}_{n,k}(q) [x]_{\downarrow_{q,F,k+1}}.
\end{eqnarray*}
Taking the coefficient of $[x]_{\downarrow_{q,F,k}}[x]_q$ on both sides 
shows that 
\begin{equation}\label{SFrec}
 \mathbf{SF}_{n+1,k}(q)=  q^{F_{k-1}}\mathbf{SF}_{n,k-1}(q)+[F_k]_q 
\mathbf{SF}_{n,k}(q)
\end{equation}
for $0 \leq k \leq n+1$.
It is then easy to check that the  $\mathbf{SF}_{n,k}(q)$s can be defined  
by the recursions (\ref{SFrec}) with the initial conditions 
that $\mathbf{SF}_{0,0}(q)=1$ and  $\mathbf{SF}_{n,k}(q)=0$ if $k < 0$ or 
$n < k$. 
A similar argument  will show that 
$\overline{\mathbf{SF}}_{n,k}(q)$ can be defined 
by the initial conditions that 
$\overline{\mathbf{SF}}_{0,0}(q)=1$ and  $\overline{\mathbf{SF}}_{n,k}(q)=0$ if $k < 0$ or 
$n <  k$ and the recursion 
\begin{equation}\label{ovSFrec}
 \overline{\mathbf{SF}}_{n+1,k}(q)=  \overline{\mathbf{SF}}_{n,k-1}(q)+[F_k]_q 
\overline{\mathbf{SF}}_{n,k}(q).
\end{equation}
for $0 \leq k \leq n+1$. 
Similarly, $\mathbf{cF}_{n,k}(q)$ can be defined 
by the initial conditions that 
$\mathbf{cF}_{0,0}(q)=1$ and  $\mathbf{cF}_{n,k}(q)=0$ if $k < 0$ or 
$n <  k$ and the recursion
\begin{equation}\label{cFrec}
 \mathbf{cF}_{n+1,k}(q)=  q^{F_{n-1}}\mathbf{cF}_{n,k-1}(q)+[F_n]_q 
\mathbf{cF}_{n,k}(q),
\end{equation}
for $0 \leq k \leq n+1$, and  
$\overline{\mathbf{cF}}_{n,k}(q)$ can be defined 
by the initial conditions that 
$\overline{\mathbf{cF}}_{0,0}(q)$ and  $\overline{\mathbf{cF}}_{n,k}(q)=0$ if $k < 0$ or 
$n < k$ and the recursion 
\begin{equation}\label{ovcFrec}
 \overline{\mathbf{cF}}_{n+1,k}(q)=  \overline{\mathbf{cF}}_{n,k-1}(q)+
[F_n]_q \overline{\mathbf{cF}}_{n,k}(q)
\end{equation}
for $0 \leq k \leq n+1$.

The main goal of this paper is to give a rook theory model 
for the polynomials $\mathbf{cF}_{n,k}(q)$, 
$\overline{\mathbf{cF}}_{n,k}(q)$, $\mathbf{SF}_{n,k}(q)$, and 
$\overline{\mathbf{SF}}_{n,k}(q)$.   Our rook theory 
model will allow us to give combinatorial proofs of 
the defining equations (\ref{cFdef}), (\ref{ovcFdef}), (\ref{SFdef}), 
and (\ref{ovSFdef}) as well as combinatorial proofs of 
the recursions (\ref{SFrec}), (\ref{ovSFrec}), (\ref{cFrec}), and (\ref{ovcFrec}).
We shall see that our rook theory model 
$\mathbf{cF}_{n,k}(q)$, 
$\overline{\mathbf{cF}}_{n,k}(q)$, $\mathbf{SF}_{n,k}(q)$, and 
$\overline{\mathbf{SF}}_{n,k}(q)$ is essentially the same 
as the the rook theory model used in \cite{BPR} to interpret 
the $\mathbf{Sf}_{n,k}(p,q)$s and $\mathbf{Sf}_{n,k}(p,q)$s 
but with a different weighting scheme.

The outline of the paper is as follows. In Section 2, we 
describe a ranking and unranking theory for the set of 
Fibonacci tilings which will a crucial element in our weighting 
scheme for our rook theory model that we shall use to give 
combinatorial interpretations of   
the polynomials $\mathbf{cF}_{n,k}(q)$, 
$\overline{\mathbf{cF}}_{n,k}(q)$, $\mathbf{SF}_{n,k}(q)$, and 
$\overline{\mathbf{SF}}_{n,k}(q)$.  In section 3, we shall 
review the rook theory model in \cite{BPR} and show how it 
can be modified for our purposes.  In Section 4, 
we shall prove general product formulas for Ferrers boards 
in our new model which will specialize 
(\ref{cFdef}), (\ref{ovcFdef}), (\ref{SFdef}), 
and (\ref{ovSFdef}) in the case where the Ferrers board 
is the staircase board whose column heights are 
$0,1, \ldots, n-1$ reading from left to right. In Section 5, 
we shall prove various special properties of the polynomials $\mathbf{cF}_{n,k}(q)$, 
$\overline{\mathbf{cF}}_{n,k}(q)$, $\mathbf{SF}_{n,k}(q)$, and 
$\overline{\mathbf{SF}}_{n,k}(q)$.

\section{Ranking and Unranking Fibonacci Tilings.}

There is a well developed theory for ranking and unranking 
combinatorial objects. See for example, Williamson's book 
\cite{Gill}. That is, give a collection of 
combinatorial objects $\mathcal{O}$ of cardinality $n$, one 
wants to define bijections  $\mathrm{rank}:\mathcal{O} \rightarrow 
\{0, \ldots, n-1\}$ and $\mathrm{unrank}:\{0, \ldots, n-1\}\rightarrow \mathcal{O}$ 
which are inverses of each other. In our case, we let 
$\mathcal{F}_n$ denote the set of Fibonnaci tilings of height $n$. Then 
we construct a tree which we call the Fibonacci tree for $F_n$. 
That is, we start from the top of a Fibonacci tiling 
and branch left if we see a tile of height 1 and branch 
right if we see a tiling of height 2. 
For example, 
the Fibonacci tree for $F_5$ is pictured 
in Figure \ref{fig:Ftree}. 

\fig{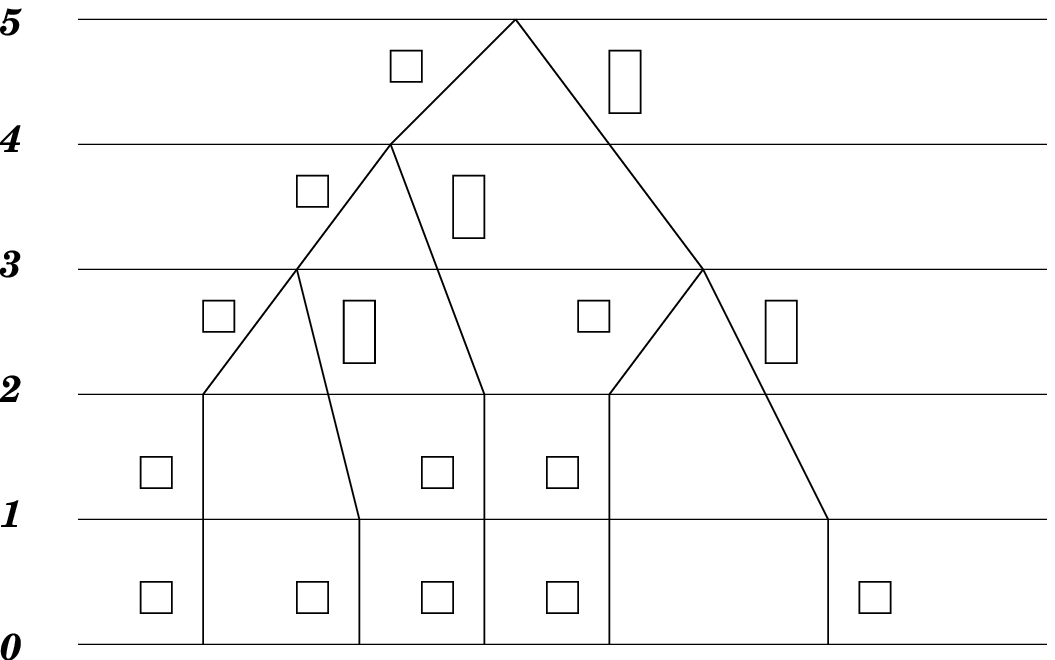}{The tree for $F_5$}

Then for any tiling $T \in \mathcal{F}_n$, we define 
the rank of $T$ for $F_n$, $\rank_n(T)$, to be the number of paths to the left of the path for $T$ in the Fibonacci tree for $F_n$. Clearly 
$$\{\rank_n(T): T \in \mathcal{F}_n\}= \{0,1,2, \ldots , F_n-1\}$$ 
so that $\sum_{T \in \mathcal{F}_n} q^{\rank_n(T)} = 1+q+ \cdots +q^{F_n-1} = 
[F_n]_q$. It is, in fact, quite easy to see compute 
the functions $\rank_n$ and $\mathrm{unrank}_n$ in this situation. 
That is, suppose that we represent the tiling $T$ as a sequence 
$seq(T) = (t_1, \ldots ,t_n)$ where reading the tiles starting at the bottom, 
$t_i = 1$ if there is a tiling $t_i$ of height $1$ that ends at level $i$ in 
$T$, $t_i =2$ if there is $t_i$ of height $2$ that ends at level $i$ in $T$, 
and $t_i =0$ if there is no tile $t_i$ that ends at level $i$ in $T$. 
For example, the tiling of $T$ height 9  pictured in Figure \ref{fig:F9}
would be represented by the sequence $seq(T)= (1,0,2,1,1,1,0,2,1)$. 

\fig{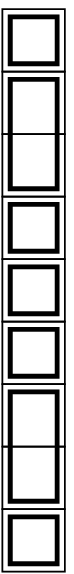}{A tiling in $\mathcal{F}_9$.}

For any statement $A$, we let  $\chi(A) =1$ is $A$ is true and $\chi(A) =0$ 
if $A$ is false. Then we have the following lemma. 

\begin{lemma} Suppose that $T \in \mathcal{F}_n$ is a Fibonacci tiling 
such that $seq(T) = (t_1, \ldots, t_n)$. Then 
$\rank_n(T) = \sum_{i=1}^n F_{i-1}\chi(t_i=2)$. 
\end{lemma}
\begin{proof}
The theorem is easy to prove by induction.  It is clearly true for 
$n=1$ and $n=2$.  Now suppose $n \geq 3$. Then it is easy 
to see from the Fibonacci tree for $F_n$ that if $t_n=2$ so that $t_{n-1} =0$, 
then the tree 
that starts at level $n-1$ which represents taking the path to the 
left starting at level $n$ is just the Fibonacci tree for $F_{n-1}$ and 
hence this tree will contain $F_{n-1}$ leaves which will all be to the left 
of path for the tiling $T$.  Then the tree that starting at level 
$n-2$ which represents taking the path to the 
right starting at level $n$ is just the Fibonacci tree for $F_{n-2}$ and 
the number of paths in this tree which lie to the left of 
the path for $T$ is just that the number of paths to the left of 
the tiling $T'$ such that $seq(T') = (t_1,\ldots,t_{n-2})$ in 
the Fibonacci tree for $F_{n-2}$.  
Thus in this case 
\begin{equation}\label{rank1}
\rank_n(T) = F_{n-1}+\rank_{n-2}(T')= 
F_{n-1}+\rank_{n-2}(t_1, \ldots, t_{n-2}).
\end{equation}
On the other hand if $t_n =1$, then we branch left at level $n$ so 
that the number of paths to the left of the path for $T$ in the 
Fibonacci tree for $F_n$ will just be the number of paths to 
the left of the tiling $T''$ such that $seq(T'') = (t_1, \ldots, t_{n-1})$ 
in the 
Fibonacci tree for $F_{n-1}$. 
Thus in this case 
\begin{equation}\label{rank2}
\rank_n(T) = \rank_{n-1}(T'')= \rank_{n-1}(t_1, \ldots, t_{n-1}).
\end{equation}
\end{proof}

For example, for the tiling $T$ in  Figure \ref{fig:F9}, 
$\rank_9(T) = F_2+F_8 = 1+21 =22$.

For the unrank function, we must rely on Zeckendorf's theorem 
\cite{Z} which states that every positive integer $n$ is uniquely represented 
as sum $n = \sum_{i=0}^k F_{c_i}$ where each $c_i \geq 2$ and 
$c_{i+1} > c_i +1$. Indeed, Zeckendorf's theorem 
says that the greedy algorithm give us the proper representation. 
That is, given $n$, find $k$ such that $F_k \leq n < F_{k+1}$, then 
the representation for $n$ is gotten by taking the representation 
for $n-F_k$ and adding $F_k$. For example, suppose that we want to find 
$T$ such that $\rank_{13}(T) =100$. Then
\begin{enumerate} 
\item $F_{11} =89 \leq 100< F_{12}=144$ so that we need to find 
the Fibonacci representation of $100-89 =11$.
\item $F_6 = 8 \leq  11 < F_7 =13$ so that we need to find the Fibonacci 
representation of $11-8 =3$. 
\item $F_4 = 3 \leq 3 < F_5 =5$. 
\end{enumerate}
Thus we can represent 
$100 =  F_4+ F_6 + F_{11} = 3+ 8 + 89$ so that 
$$seq(T) =(1,1,1,0,2,0,2,1,1,1,0,2,1).$$

\section{The rook theory model for the 
$\mathbf{SF}_{n,k}(q)$s and the $\mathbf{cF}_{n,k}(q)$s.}

In this section, we shall give a rook theory model which 
will allow us to give combinatorial interpretations for 
the $\mathbf{SF}_{n,k}(q)$s and the $\mathbf{cF}_{n,k}(q)$s. 
This rook theory 
model is based on the one 
which Bach, Paudyal, and Remmel used in \cite{BPR} to give combinatorial interpretations to 
the $\mathbf{Sf}_{n,k}(p,q)$s and the $\mathbf{cf}_{n,k}(p,q)$s.  
Thus, we shall briefly review the rook theory model in \cite{BPR}.

A Ferrers board $B=F(b_1, \ldots, b_n)$ is 
a board whose column heights are $b_1, \ldots, b_n$, reading 
from left to right, such that $0\leq b_1 \leq b_2 \leq \cdots \leq b_n$. 
We shall let $B_n$ denote the Ferrers board $F(0,1, \ldots, n-1)$. 
For example, the Ferrers board $B = F(2,2,3,5)$ is 
pictured on the left of Figure \ref{fig:Ferrers} 
and the Ferrers board $B_4$ is pictured on the right of 
Figure \ref{fig:Ferrers} 

\fig{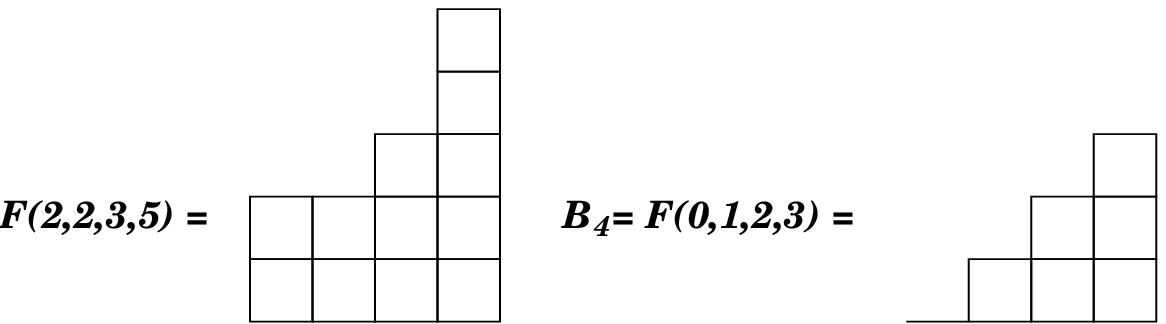}{Ferrers boards.}

Classically, there are two type of rook placements that we 
consider on a Ferrers board $B$.  First we let 
$\mathcal{N}_k(B)$ be the set of all placements of 
$k$ rooks in $B$ such that no two rooks lie in the same 
row or column.  We shall call an element  of 
$\mathcal{N}_k(B)$ a placement of $k$ non-attacking rooks 
in $B$ or just a rook placement for short.  We let 
$\mathcal{F}_k(B)$ be the set of all placements of 
$k$ rooks in $B$ such that no two rooks lie in the same 
column.  We shall call an element  of 
$\mathcal{F}_k(B)$ a file placement of $k$ rooks 
in $B$.   
Thus file placements differ from rook placements 
in that file placements allow 
two rooks to be in the same row.  For example, 
we exhibit a placement of 3 non-attacking rooks 
in $F(2,2,3,5)$ on the left in Figure 
\ref{fig:placements} and a file placement of 3 rooks on 
the right in  Figure \ref{fig:placements}.

\fig{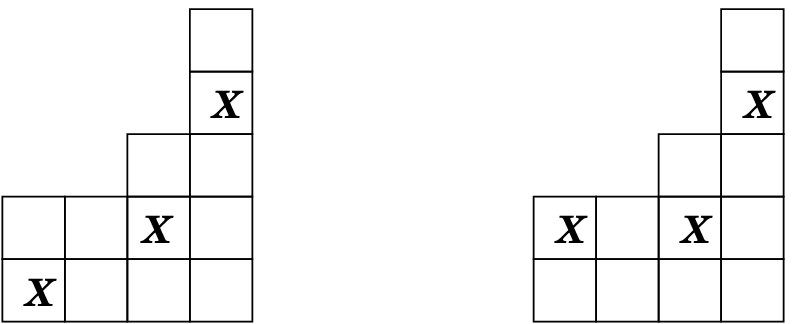}{Examples of rook and file placements.}

Given a Ferrers board $B = F(b_1, \ldots, b_n)$, we 
define the $k$-th rook number of $B$ to be 
$r_k(B) = |\mathcal{N}_k(B)|$ and the $k$-th file number 
of $B$ to be $f_k(B) = |\mathcal{F}_k(B)|$. Then the rook 
theory interpretation of the classical Stirling numbers is 
\begin{eqnarray*}
S_{n,k} &=& r_{n-k}(B_n) \ \mbox{for all}\ 1 \leq k \leq n \ \mbox{and} \\
c_{n,k} &=& f_{n-k}(B_n) \ \mbox{for all}\ 1 \leq k \leq n.
\end{eqnarray*}

The idea of \cite{BPR} is to modify the sets $\mathcal{N}_k(B)$ and 
$\mathcal{F}_k(B)$ to replace rooks with Fibonacci tilings. 
The analogue of file placements is very straightforward. 
That is, if $B=F(b_1, \ldots, b_n)$, then we let 
$\mathcal{FT}_k(B)$ denote the set of all configurations such that 
there are $k$ columns $(i_1, \ldots, i_k)$ of $B$ where 
$1 \leq i_1 < \cdots < i_k \leq n$ such that in each 
column $i_j$, we have placed one of the tilings $T_{i_j}$ for the Fibonacci 
number $F_{b_{i_j}}$.  We shall call such a configuration 
a Fibonacci file placement and denote it by 
$$P = (({i_1},T_{i_1}),  \ldots, ({i_k},T_{i_k})).$$
 Let 
$\mathrm{one}(P)$ denote the number of tiles of height 1 that appear 
in $P$ and $\mathrm{two}(P)$ denote the number of tiles of height 2 that appear 
in $P$.  Then in \cite{BPR}, we defined the weight of $P$, $WF(P,p,q)$, to be 
$q^{\mathrm{one}(P)}p^{\mathrm{two}(P)}$.  For example, we have 
pictured an element $P$ of $\mathcal{FT}_3(F(2,3,4,4,5))$ in 
Figure \ref{fig:Fibfile} whose weight is $q^7 p^2$. Then 
we defined the $k$-th $p,q$-Fibonacci file polynomial of $B$, $\mathbf{fT}_k(B,p,q)$, 
by setting 
\begin{equation*}
 \mathbf{fT}_k(B,p,q) = \sum_{P \in \mathcal{FT}_k(B)} WF(P,p,q).
\end{equation*}
If $k =0$, then the only element of $\mathcal{FT}_k(B)$ is the empty placement 
whose weight by definition is 1.

\fig{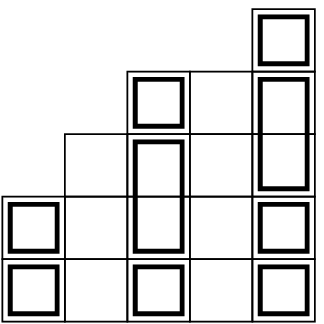}{A Fibonacci file placement.}

Then in \cite{BPR}, we proved the following theorem concerning 
Fibonacci file placements in Ferrers boards. 

\begin{theorem}\label{thm:Ferrersfilerec} 
Let $B =F(b_1, \ldots, b_n)$ be a Ferrers 
board where $0 \leq b_1 \leq \cdots \leq b_n$ and $b_n > 0$. 
Let $B^- = F(b_1, \ldots, b_{n-1})$. Then for all 
$1 \leq k \leq n$, 
\begin{equation}\label{Fibtilerec}
\mathbf{fT}_k(B,p,q) = \mathbf{fT}_{k}(B^-,p,q)+ F_{b_n}(p,q) \mathbf{fT}_{k-1}(B^-,p,q).
\end{equation}
\end{theorem}

To obtain the $q$-analogues that we desire for this paper, we 
define a new weight functions for elements of $\mathcal{FT}_k(B)$ where 
$B=F(b_1, \ldots, b_n)$ is Ferrers board. That is given 
a Fibonacci file placement 
$P = (({i_1},T_{i_1}),  \ldots, ({i_{n-k}},T_{i_{n-k}}))$ in 
$\mathcal{FT}_{n-k}(B)$, let $({j_1}, \ldots, {j_k})$ be 
the sequence of columns in $B$ which have no tilings, reading 
from left to right. Then we define 
\begin{eqnarray*}
\mathbf{w_{B,q}}(P) &=& q^{\sum_{s=1}^{n-k} \rank_{b_{i_s}}(T_{i_s}) + 
\sum_{t=1}^k F_{b_{j_t}}} \ \mbox{and} \\
\overline{\mathbf{w_{B,q}}}(P) &=& 
q^{\sum_{s=1}^{n-k} \rank_{b_{i_s}}(T_{i_s})}
\end{eqnarray*}
Note that the only difference between these two weight functions 
is that if $b_i$ is column that does not contain a tiling in 
$P$, then it contributes a factor of 
$q^{F_{b_i}}$ to $\mathbf{w_{B,q}}(P)$ and a factor of 1 to 
$\overline{\mathbf{w_{B,q}}}(P)$.
We then define $\mathbf{FT}_k(B,q)$ and $\overline{\mathbf{FT}}_k(B,q)$, 
by setting 
\begin{eqnarray*}
 \mathbf{FT}_k(B,q) &=& \sum_{P \in \mathcal{FT}_k(B)} \mathbf{w_{B,q}}(P) 
\ \mbox{and} \\
\overline{\mathbf{FT}}_k(B,q) &=& \sum_{P \in \mathcal{FT}_k(B)} 
\overline{\mathbf{w_{B,q}}}(P).
\end{eqnarray*}
If $k =0$, then the only element of $\mathcal{FT}_k(B)$ is the empty 
placement $\emptyset$ so that  
$\mathbf{w_{B,q}}(\emptyset) =q^{\sum_{i=1}^n F_{b_i}}$ and 
$\overline{\mathbf{w_{B,q}}}(\emptyset) =1$.

Then we have the following analogue of Theorem 
\ref{thm:Ferrersfilerec}.

\begin{theorem}\label{thm:qFerrersfilerec} 
Let $B =F(b_1, \ldots, b_n)$ be a Ferrers 
board where $0 \leq b_1 \leq \cdots \leq b_n$ and $b_n > 0$. 
Let $B^- = F(b_1, \ldots, b_{n-1})$. Then for all 
$1 \leq k \leq n$, 
\begin{equation}\label{qFibtilerec}
\mathbf{FT}_k(B,q) = q^{F_{b_n}}\mathbf{FT}_{k}(B^-,q)+ [F_{b_n}]_q \mathbf{FT}_{k-1}(B^-,p,q)
\end{equation}
and 
\begin{equation}\label{ovqFibtilerec}
\overline{\mathbf{FT}}_k(B,q) = \overline{\mathbf{FT}}_{k}(B^-,q)+ [F_{b_n}]_q \mathbf{FT}_{k-1}(B^-,p,q).
\end{equation}
\end{theorem}
\begin{proof} 
We claim (\ref{qFibtilerec}) 
results by classifying the Fibonacci file placements 
in $\mathcal{FT}_k(B)$ according to whether there is a tiling in the 
last column. If there is no tiling in the last column of $P$, 
then removing the last column of  $P$ produces 
 an element of $\mathcal{FT}_k(B^-)$ . Thus such placements 
contribute $q^{F_{b_n}}\mathbf{FT}_{k}(B^-,q)$ to $\mathbf{FT}_{k}(B,q)$ 
since the fact that the last column has no tiling means that 
it contributes 
a factor of $q^{F_{b_n}}$ to $\mathbf{w_{B,q}}(P)$.  
If there is a tiling in 
the last column, then the Fibonacci file placement  
that results by removing the last column is an 
element of $\mathcal{FT}_{k-1}(B^-)$ and the sum of 
the weights of the possible 
Fibonacci tilings of height $b_n$ for the last column 
is $\sum_{T \in \mathcal{F}_{b_n}}q^{\rank_{b_n}(T)} = 
[F_{b_n}]_q$. Hence such placements 
contribute $[F_{b_n}]_q \mathbf{FT}_{k-1}(B^-,q)$ to 
$\mathbf{FT}_{k}(B,q)$. Thus 
$$\mathbf{FT}_k(B,q) = q^{F_{b_n}}\mathbf{FT}_{k}(B^-,q)+ [F_{b_n}]_q \mathbf{FT}_{k-1}(B^-,p,q).$$

A similar argument will prove (\ref{ovqFibtilerec}). 
\end{proof}

If $B=F(b_1, \ldots, b_n)$ is a Ferrers board, 
then we let $B_x$ denote the board that results by 
adding $x$ rows of length $n$ below $B$.  We label 
these rows from top to bottom with the numbers 
$1,2, \ldots, x$. We shall call  
the line that separates $B$ from these $x$ rows the {\em bar}. 
A mixed file placement $P$ on the board $B_x$ consists 
of picking for each column $b_i$ either (i) a Fibonacci tiling 
$T_i$ of height $b_i$ above the bar or (ii) picking 
a row $j$ below the bar to place a rook in the cell in row $j$ 
and column $i$.  Let $\mathcal{M}_n(B_x)$ denote set of all 
mixed rook placements on $B$. For any $P \in \mathcal{M}_n(B_x)$, 
we let $\mathrm{one}(P)$ denote the number of tiles of height 1 that appear 
in $P$ and $\mathrm{two}(P)$ denote the set tiles of height 2 that appear 
in $P$.  Then in \cite{BPR}, we defined the weight of $P$, $WF(P,p,q)$, to be 
$q^{\mathrm{one}(P)}p^{\mathrm{two}(P)}$.
For example, 
Figure \ref{fig:mixed} pictures a mixed placement $P$ in 
$B_x$ where $B = F(2,3,4,4,5,5)$ and $x$ is 9 such that 
$WF(P,p,q) =  q^7p^2$. 

\fig{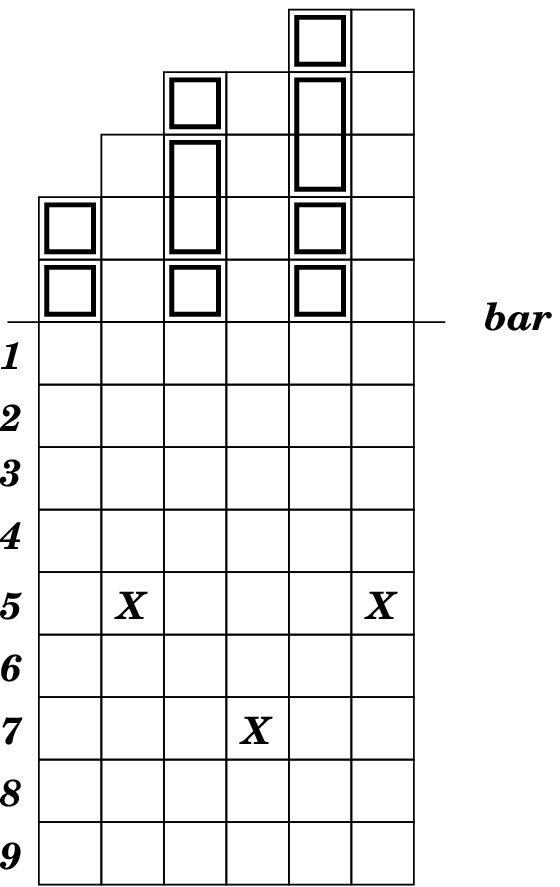}{A mixed file placement.}

Also in \cite{BPR}, we proved the following theorem by counting  
$\sum_{P \in \mathcal{M}_n(B_x)} WF(P,p,q)$ in two different ways. 

\begin{theorem}\label{thm:Ferrersfileprod}
Let $B =F(b_1, \ldots, b_n)$ be a Ferrers 
board where $0 \leq b_1 \leq \cdots \leq b_n$ and $b_n > 0$. 
\begin{equation}\label{Fibfileprod}
(x+F_{b_1}(p,q))(x+F_{b_2}(p,q)) \cdots (x+F_{b_n}(p,q)) =
\sum_{k=0}^n \mathbf{fT}_k(B,p,q) x^{n-k}.
\end{equation}
\end{theorem}

To obtain the desired $q$-analogues for this paper, we 
must define new weight functions 
for mixed placements $P \in \mathcal{M}_n(B_x)$. 
That is, suppose that $P \cap B$ is the Fibonacci tile placement 
$Q = (({i_1},T_{i_1}),  \ldots, ({i_k},T_{i_{n-k}}))$, and suppose that, for the rooks below the bar in columns 
$1 \leq {j_1}<  \ldots {j_k} \leq n$, the rook in column 
$j_s$ is in row $d_{j_s}$ for $s =1, \ldots, k$.  
Then we define
\begin{eqnarray*}
\mathbf{w_{B_x,q}}(P) &=& \mathbf{w_{B,q}}(P)q^{\sum_{t=1}^k d_{j_t} -1} = 
q^{\sum_{s=1}^{n-k} 
\rank_{b_{i_s}}(T_{i_s}) + \sum_{t=1}^k F_{b_{j_t}}+d_{j_t}-1} \ \mbox{and} \\
\overline{\mathbf{w_{B_x,q}}}(P) &=& \overline{\mathbf{w_{B,q}}}(P)q^{\sum_{t=1}^k d_{j_t} -1} = 
q^{\sum_{s=1}^{n-k} 
\rank_{b_{i_s}}(T_{i_s}) + \sum_{t=1}^k d_{j_t}-1}.
\end{eqnarray*}
That is, for each column $i$ the choice of a Fibonacci tiling 
$T_i$ of height $b_i$ above the bar contributes a factor 
of $q^{\rank_{b_i}(T_i)}$ to $\mathbf{w_{B_x,q}}(P)$ and the 
choice of picking 
a row $j$ below the bar to place a rook in the cell in row $j$ 
and column $i$ contributes a factor of 
$q^{F_{b_i}+j-1}$ to $\mathbf{w_{B_x,q}}(P)$. Similarly, 
for each column $b_i$ the choice of a Fibonacci tiling 
$T_i$ of height $b_i$ above the bar contributes a factor 
of $q^{\rank_{b_i}(T_i)}$ to $\overline{\mathbf{w_{B_x,q}}}(P)$ and the 
choice of picking 
a row $j$ below the bar to place a rook in the cell in row $j$ 
and column $i$ contributes a factor of 
$q^{j-1}$ to $\overline{\mathbf{w_{B_x,q}}}(P)$.

Then we have the following analogue of Theorem \ref{thm:Ferrersfileprod}.

\begin{theorem}\label{thm:qFerrersfileprod}
Let $B =F(b_1, \ldots, b_n)$ be a Ferrers 
board where $0 \leq b_1 \leq \cdots \leq b_n$ and $b_n > 0$. 
Then for all positive integers $x$, 
\begin{equation}\label{qFibfileprod}
[x+F_{b_1}]_q [x+F_{b_2}]_q \cdots [x+F_{b_n}]_q =
\sum_{k=0}^n \mathbf{FT}_k(B,q) [x]_q^{n-k} 
\end{equation}
and 
\begin{equation}\label{ovqFibfileprod}
([x]_q+[F_{b_1}]_q) ([x]_q+[F_{b_2}]_q) \cdots ([x]_q+[F_{b_n}]_q) =
\sum_{k=0}^n \overline{\mathbf{FT}}_k(B,q) [x]_q^{n-k} 
\end{equation}
\end{theorem}
\begin{proof}
To prove (\ref{qFibfileprod}),   
fix $x$ to be a positive integer and consider 
the sums 
\begin{eqnarray*}
S&=&\sum_{P \in \mathcal{M}_n(B_x)} \mathbf{w_{B_x,q}}(P) \ \mbox{and} \\ 
\overline{S} &=&\sum_{P \in \mathcal{M}_n(B_x)} 
\overline{\mathbf{w_{B_x,q}}}(P). 
\end{eqnarray*}

For $S$, in  a given column $i$, our choice of the Fibonacci tiling 
of height $b_i$ will contribute a factor 
of $\sum_{T \in \mathcal{F}_n}q^{\rank_{b_i}(T)} =[F_{b_i}]_q$ 
to $S$.  Our choice of placing a rook below the bar in 
column $i$ contribute a factor of 
$$\sum_{j=1}^x q^{F_{b_i}+j-1} = q^{F_{b_i}}(1+q+q^2 + \cdots q^{x-1}) =
q^{F_{b_i}}[x]_q$$ 
to $S$.  As $[F_{b_i}]_q + q^{F_{b_i}}[x]_q =[x+F_{b_i}]_q$,   
each column of $b_i$ of $B$ contributes 
a factor of $[x+ F_{b_i}]_q$ to $S$ so that 
$$S = \prod_{i=1}^n [x +  F_{b_i}]_q.$$

For $\overline{S}$, in  a given column $i$, our choice of the Fibonacci tiling 
of height $b_i$ will contribute a factor 
of $\sum_{T \in \mathcal{F}_n}q^{\rank_{b_i}(T)} =[F_{b_i}]_q$ 
to $S$.  Our choice of placing a rook below the bar in 
column $i$ contribute a factor of 
$$\sum_{j=1}^x q^{j-1} = [x]_q$$ 
to $\overline{S}$. Thus each column $b_i$ contributes 
a factor of $[x]_q+ [F_{b_i}]_q$ to $S$ so that 
$$S = \prod_{i=1}^n ([x]_q + [F_{b_i}]_q).$$

On the other 
hand, suppose that we fix a Fibonacci file placement  
$P \in \mathcal{FT}_k(B)$. 
Then we want to compute $S_P = \sum_{Q  \in \mathcal{M}_n(B), 
Q \cap B = P} \mathbf{w_{B_x,q}}(Q)$ which is the sum of 
$\mathbf{w_{B_x,q}}(Q)$ over all 
mixed placements $Q$ such that $Q$ intersect $B$ equals $P$. 
It it easy to see that such a $Q$ arises by choosing 
a rook to be placed below the bar for each column 
that does not contain a tiling.  Each such column contributes 
a factor of $1+q+ \cdots +q^{x-1} =[x]_q$ in addition 
to the weight $\mathbf{w_{B,q}}(P)$. Thus it follows that 
$\displaystyle S_P =\mathbf{w_{B,q}}(P) [x]_q^{n-k}$. 
Hence it follows that 
\begin{eqnarray*}
S &=&  \sum_{k=0}^n \sum_{P \in \mathcal{FT}_k(B)} S_P \\
&=& \sum_{k=0}^n [x]_q^{n-k} \sum_{P \in \mathcal{FT}_k(B)} 
\mathbf{w_{B,q}}(P) \\
&=& \sum_{k=0}^n \mathbf{FT}_k(B,q) \ [x]_q^{n-k}.
\end{eqnarray*}

The same argument will show that 
\begin{equation*}
\overline{S} = \sum_{k=0}^n \overline{\mathbf{FT}}_k(B,q) \ [x]_q^{n-k}.
\end{equation*}

\end{proof}

Now consider the special case of the previous two theorems 
when $B_n = F(0,1,2, \ldots, n-1)$. Then (\ref{qFibtilerec}) implies 
that 
$$\mathbf{FT}_{n+1-k}(B_{n+1},q) = q^{F_n}\mathbf{FT}_{n+1-k}(B_n,p,q) + 
[F_n]_q 
\mathbf{FT}_{n-k}(B_n,q).$$
It then easily follows that for all $0 \leq k \leq n$, 
\begin{equation}\label{cnkpqinterp}
 \mathbf{cF}_{n,k}(q) = \mathbf{FT}_{n-k}(B_n,q).
\end{equation}
Note that $\mathbf{cF}_{n,0}(q) = 0$ for all $n \geq 1$ since 
there are no Fibonacci file placements in 
$\mathcal{FT}_n(B_n)$ since there are only $n-1$ non-zero columns. 
Moreover such a situation, we see that (\ref{cnkpqinterp}) 
implies that 
\begin{equation*}
[x]_q [x+F_1]_q [x+F_2]_q \cdots [x+F_{n-1}]_q = 
\sum_{k=1}^n \mathbf{cF}_{n,k}(q) \ [x]_q^k.
\end{equation*}
Thus we have given a combinatorial proof of 
(\ref{cFdef}).

Similarly (\ref{ovqFibtilerec}) implies 
that 
$$\overline{\mathbf{FT}}_{n+1-k}(B_{n+1},q) = \overline{\mathbf{FT}}_{n+1-k}(B_n,p,q) + 
[F_n]_q 
\overline{\mathbf{FT}}_{n-k}(B_n,q).$$
It then easily follows that for all $0 \leq k \leq n$, 
\begin{equation}\label{Snkpqinterp}
 \overline{\mathbf{cF}}_{n,k}(q) = \overline{\mathbf{FT}}_{n-k}(B_n,q).
\end{equation}
Moreover such a situation, we see that (\ref{Snkpqinterp}) 
implies that 
\begin{equation*}
[x]_q ([x]_q+[F_1]_q) ([x]_q+[F_2]_q) \cdots ([x]_q+[F_{n-1}]_q) = 
\sum_{k=1}^n \overline{\mathbf{cF}}_{n,k}(q) \ [x]_q^k.
\end{equation*}
Thus we have given a combinatorial proof of 
(\ref{ovcFdef}).

The Fibonacci analogue of rook  placements defined in 
\cite{BPR} is a slight 
variation of Fibonacci file placements. The main 
difference is that each 
tiling will cancel some of the top most cells in each column 
to its right that has not been canceled by a tiling 
which is further to the left.  Our goal is to ensure 
that if we start with a Ferrers board $B =F(b_1, \ldots, b_n)$, 
our cancellation scheme will ensure that the number of 
uncanceled cells in the empty columns are $b_1, \ldots, b_{n-k}$, 
reading from left to right. 
That is, if $B=F(b_1, \ldots, b_n)$, then we let 
$\mathcal{NT}_k(B)$ denote the set of all configurations such that 
that there are $k$ columns $(i_1, \ldots, i_k)$ of $B$ where 
$1 \leq i_1 < \cdots < i_k \leq n$ such that the 
following conditions hold. \begin{itemize}
\item[1.] In column ${i_1}$, we place a Fibonacci tiling 
$T_{i,1}$ of height $b_{i_1}$ and for each $j > i_1$, 
this tiling cancels the top $b_j-b_{j-1}$ cells at the top 
of column $j$. This cancellation has the effect of 
ensuring that the number of uncanceled cells in the columns 
without tilings at this point is 
$b_1, \ldots, b_{n-1}$, reading from left to right. 
\item[2.] In column ${i_2}$, our cancellation due to the tiling 
in column $i_1$ ensures that there are $b_{i_2-1}$ uncanceled 
cells in column $i_2$. Then we place a Fibonacci tiling 
$T_{i,2}$ of height $b_{i_2-1}$ and for each $j > i_2$, 
we cancel the top $b_{j-1}-b_{j-2}$ cells in column $j$ 
that has not been canceled by the tiling in column $i_1$. 
This cancellation has the effect of 
ensuring that the number of uncanceled cells in columns 
without tilings at this point 
is $b_1, \ldots, b_{n-2}$, reading from left to right.
\item[3.] In general, when we reach column $i_s$, we assume 
that the cancellation due to the tilings in columns  
$i_1, \ldots, i_{j-1}$ ensure that the number of uncanceled 
cells in the columns without tilings is $b_1, \ldots, b_{n-(s-1)}$, 
reading from left to right. Thus there will be 
$b_{i_s -(s-1)}$ uncanceled cells in column $i_s$ at this point. 
Then we place a Fibonacci tiling 
$T_{i,s}$ of height $b_{i_s-(s-1)}$ and for each $j > i_s$, 
this tiling will 
cancel the top $b_{j-(s-1)}-b_{j-s}$ cells in column 
$j$ that has not been canceled by the tilings in 
columns $i_1, \ldots, i_{s-1}$. 
This cancellation has the effect of 
ensuring that the number of uncanceled cells in columns 
without tilings at this point 
is $b_1, \ldots, b_{n-s}$, reading from left to right.
\end{itemize}

We shall call such a configuration 
a Fibonacci rook placement and denote it by 
$$P = (({i_1},T_{i_1}),  \ldots, ({i_k},T_{i_k})).$$ 
Let $\mathrm{one}(P)$ denote the number of tiles of height 1 that appear 
in $P$ and $\mathrm{two}(P)$ denote the number of tiles of height 2 that appear 
in $P$.  Then in \cite{BPR}, we  defined the weight of $P$, $WF(P,p,q)$, to be 
$q^{\mathrm{one}(P)}p^{\mathrm{two}(P)}$.  For example, on the left in 
Figure \ref{fig:Fibrook}, we have 
pictured an element $P$ of $\mathcal{NT}_3(F(2,3,4,4,6,6))$ 
whose weight is $q^5 p^2$. In 
Figure \ref{fig:Fibrook}, we have 
indicated the canceled cells by the tiling in column 
$i$ by placing an $i$ in the cell.  
We note in the special case where $B = F(0,k,2k, \ldots, (n-1)k)$, then 
our cancellation scheme is quite simple. That is, each tiling 
just cancels the top $k$ cells in each column to its right which 
has not been canceled by tilings to its left. 
For example, on the right in 
Figure \ref{fig:Fibrook}, we have 
pictured an element $P$ of $\mathcal{NT}_3(F(0,1,2,3,4,5))$ 
whose weight is $q^6 p$. Again, we have 
indicated the canceled cells by the tiling in column 
$i$ by placing an $i$ in the cell.

\fig{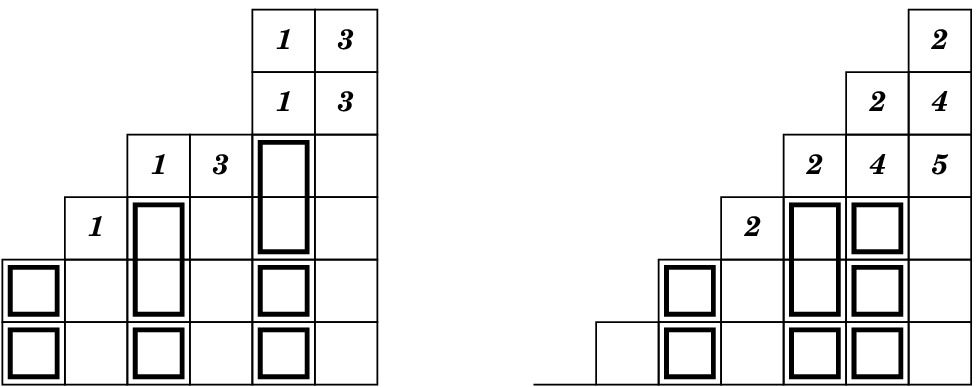}{A Fibonacci rook placement.}

We define the $k$-th $p,q$-Fibonacci rook polynomial of $B$, $\mathbf{rT}_k(B,p,q)$, 
by setting 
\begin{equation*}
 \mathbf{rT}_k(B,p,q) = \sum_{P \in \mathcal{NT}_k(B)} WF(P,p,q).
\end{equation*}
If $k =0$, then the only element of $\mathcal{FT}_k(B)$ is the empty placement 
whose weight by definition is 1.

Then in \cite{BPR}, we proved the following two theorems concerning 
Fibonacci rook placements in Ferrers boards. 

\begin{theorem}\label{thm:Ferrersrookrec} 
Let $B =F(b_1, \ldots, b_n)$ be a Ferrers 
board where $0 \leq b_1 \leq \cdots \leq b_n$ and $b_n > 0$. 
Let $B^- = F(b_1, \ldots, b_{n-1})$. Then for all 
$1 \leq k \leq n$, 
\begin{equation}\label{Fibrookrec}
\mathbf{rT}_k(B,p,q) = \mathbf{rT}_{k}(B^-,p,q)+ F_{b_{n-(k-1)}}(p,q) \mathbf{rT}_{k-1}(B^-,p,q).
\end{equation}
\end{theorem}

\begin{theorem}\label{thm:Ferrersrookprod} 
Let $B =F(b_1, \ldots, b_n)$ be a Ferrers 
board where $0 \leq b_1 \leq \cdots \leq b_n$ and $b_n > 0$. 
\begin{equation}\label{Fibrookprod}
x^n =
\sum_{k=0}^n \mathbf{rT}_{n-k}(B,p,q) (x-F_{b_1}(p,q))(x-F_{b_2}(p,q))\cdots 
(x-F_{b_{k}}(p,q)).
\end{equation}
\end{theorem}

To obtain the $q$-analogues that we want for this paper, we need to define two 
new weight functions on Fibonacci rook tilings. 
That is, suppose that $B =F(b_1, \ldots, b_n)$ is a Ferrers board 
and $P = (({i_1},T_{i_1}),  \ldots, ({i_k},T_{i_k}))$ is an Fibonacci 
rook tiling in $\mathcal{NT}_k(B)$. Then we know 
that the number of uncanceled cells in the $n-k$ columns 
which do not have tilings are $b_1, \ldots ,b_{n-k}$ reading from 
left to right. Suppose that the number of uncanceled cells in 
the columns with tilings are $e_1, \ldots, e_k$ reading from left 
to right so that tiling $T_{i_j}$ is of height $e_j$ for $j =1, 
\ldots, k$.  The we define 
\begin{eqnarray*}
\mathbf{W_{B,q}}(P) &=& q^{\sum_{s=1}^k \rank_{e_s}(T_{i_s}) + 
\sum_{t=1}^{n-k} F_{b_t}} \ \mbox{and} \\
\overline{\mathbf{W_{B,q}}}(P) &=& q^{\sum_{s=1}^k \rank_{e_s}(T_{i_s})}.
\end{eqnarray*} 
For example, if $B=(2,3,4,4,5,5)$ and $P=((1,T_1),(3,T_3),(5,T_5))$ is the rook tiling pictured 
in Figure \ref{fig:Fibrook}, then $e_1 =2$, $e_2 = 3$ and $e_3 =4$ and 
one can check that $\rank_2(T_1) =0$, $\rank_3(T_3) = F_2 =1$, and \
$\rank_4(T_5) = F_3 = 2$. Thus 
$\mathbf{W_{B,q}}(P)=q^{0+1+2+F_2+F_3+F_4} = q^{9}$ and 
$\overline{\mathbf{W_{B,q}}}(P)=q^{0+1+2} = q^{3}$.
If $k =0$, then the only element of $\mathcal{FT}_k(B)$ is the empty placement 
$\emptyset$ which means that $\mathbf{W_{B,q}}(\emptyset) = 
q^{\sum_{i=1}^n F_{b_i}}$ and $\mathbf{W_{B,q}}(\emptyset) =1$.

Then we define $\mathbf{RT}_k(B,q)$  
by setting 
\begin{equation*}
 \mathbf{RT}_k(B,q) = \sum_{P \in \mathcal{NT}_k(B)} \mathbf{W_{B,q}}(P)
\end{equation*}
and 
\begin{equation*}
 \overline{\mathbf{RT}}_k(B,q) = \sum_{P \in \mathcal{NT}_k(B)} 
\overline{\mathbf{W_{B,q}}}(P).
\end{equation*}
Note that because of our cancellation scheme, there is a very 
simple relationship between $\mathbf{RT}_k(B,q)$ and 
$\overline{\mathbf{RT}}_k(B,q)$ in the case where $B = F(b_1, \ldots, 
b_n)$. That is, in any placement $P \in \mathcal{NT}_k(B)$, 
the empty columns have $b_1, \ldots, b_{n-k}$ uncanceled cells, 
reading from left to right, so that 
\begin{equation}\label{rel} 
\mathbf{RT}_k(B,q) = q^{\sum_{i=1}^{n-k} F_{b_i}}\overline{\mathbf{RT}}_k(B,q).
\end{equation}

Let $B=F(b_1, \ldots, b_n)$ be  a Ferrers board and 
$x$ be a positive integer. 
Then we let $AugB_x$ denote the board where 
we start with $B_x$ and add the flip of the board $B$ about 
its baseline below the board. We shall call the 
the line that separates $B$ from these $x$ rows the {\em upper bar} 
and the line that separates the $x$ rows from the flip 
of $B$ added below the $x$ rows the {\em lower bar}. We shall 
call the flipped version of $B$ added below $B_x$ the board 
$\overline{B}$. For example, 
if $B=F(2,3,4,4,5,5)$, then the board $AugB_7$ is pictured 
in Figure \ref{fig:aug}.

\fig{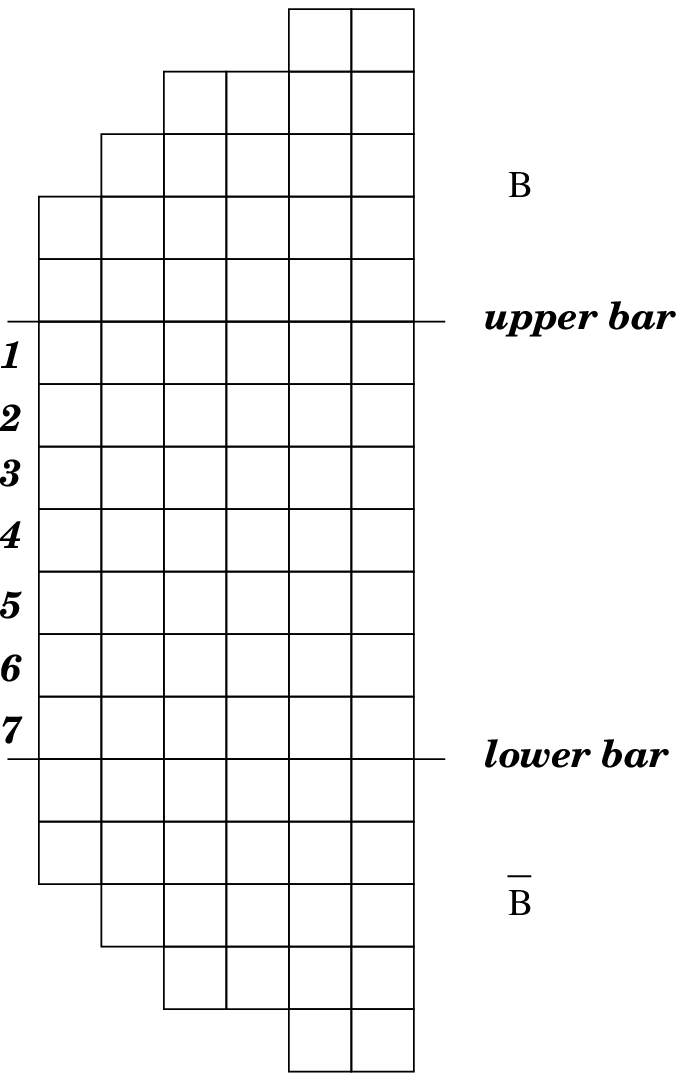}{An example of an augmented board $AugB_x$.}

The analogue of mixed placements in $AugB_x$ are 
more complex than the mixed placements for $B_x$. We process 
the columns from left to right. 
If we are in column 1, then we can do one of the following three things. \begin{itemize}
\item[i.] We can put a Fibonacci tiling in  
cells in the first column in $B$.  Then we must 
cancel the top-most cells in each of the columns in $B$ to its right 
so that the number of uncanceled cells in 
the columns to its right are $b_1,b_2, \ldots, b_{n-1}$, respectively, as 
we read from left to right. This means 
that we will cancel $b_i-b_{i-1}$ at the top of column $i$ in $B$ 
for $i=2, \ldots, n$. We also cancel the same number 
of cells at the bottom of the corresponding columns of $\overline{B}$.
\item[ii.] We can place a rook in any row of column $1$ that 
lies between the upper bar and lower bar.  This rook 
will not cancel anything. 
\item[iii.] We can put a flip of Fibonacci tiling 
in column $1$ of $\overline{B}$.  This tiling will not 
cancel anything. \end{itemize}

Next assume that when we get to column $j$, the 
number of uncanceled cells in the columns that have 
no tilings in $B$ and $\overline{B}$ are $b_1, \ldots, b_k$  for some $k$ as 
we read from left to right. Suppose there are 
$b_i$ uncanceled cells in $B$ in column $j$.  
Then we can do one of three things. \begin{itemize}
\item[i.] We can put a Fibonacci tiling of height $b_i$  in the uncanceled cells in column $j$ in $B$.  Then we must 
cancel top-most cells of the columns in $B$ to its right 
so that the number of uncanceled cells in 
the columns which have no tilings up to this point 
are $b_1,b_2, \ldots, b_{k-1}$, 
We also cancel the same number 
of cells at the bottom of the corresponding columns of $\overline{B}$
\item[ii.] We can place a rook in any row of column $j$ that 
lies between the upper bar and lower bar.  This rook 
will not cancel anything. 
\item[iii.] We can put a flip of Fibonacci tiling in the $b_i$ 
uncanceled cells in column $j$ of $\overline{B}$.  This tiling will not 
cancel anything \end{itemize}

We let $\mathcal{M}_n(AugB_x)$ denote set of all 
mixed rook placements on $AugB_x$. For any placement 
$P \in \mathcal{M}_n(AugB_x)$, 
we define $\mathbf{W_{AugB_x,q}}(P)$ and $\overline{\mathbf{W_{AugB_x,q}}}(P)$
as follows. For any column $i$, suppose that 
the number of uncanceled cells in $B$ in column $i$ is $t_i$. 
Then the factor $\mathbf{W_{i,AugB_x,q}}(P)$ that the placement 
in column $i$ contributes to $\mathbf{W_{AugB_x,q}}(P)$ is 
\begin{enumerate}
\item $q^{\rank_{t_i}(T_i)}$ if there is tiling $T_i$ in $B$ in 
column $i$, 
\item $q^{F_{t_i}+s_i-1}$ if there is a rook in row $s_i^{th}$ row from 
the top in the $x$ rows that lie between the upper bar and lower bar, 
and 
\item $-q^{\rank_{t_i}(T_i)}$ if there is a flip of a tiling $T_i$ 
in column $i$ of $\overline{B}$.
\end{enumerate}
Then we define 
$$\mathbf{W_{AugB_x,q}}(P) = \prod_{i=1}^n \mathbf{W_{i,AugB_x,q}}(P).$$

Similarly, the factor $\overline{\mathbf{W_{i,AugB_x,q}}}(P)$ 
that the tile placement 
in column $i$ contributes to $\overline{\mathbf{W_{AugB_x,q}}}(P)$ is  
\begin{enumerate}
\item $q^{\rank_{t_i}(T_i)}$ if there is tiling $T_i$ in $B$ in 
column $i$, 
\item $q^{s_i-1}$ if there is a rook in row $s_i^{th}$ row from 
the top in the $x$ rows that lie between the upper bar and lower bar, 
and 
\item $-q^{\rank_{t_i}(T_i)}$ if there is a flip of a tiling $T_i$ 
in column $i$ of $\overline{B}$.
\end{enumerate}
Then we define 
$$\overline{\mathbf{W_{AugB_x,q}}}(P) = \prod_{i=1}^n 
\overline{\mathbf{W_{i,AugB_x,q}}}(P).$$

For example, 
Figure \ref{fig:aug2} pictures a mixed placement $P$ in 
$AugB_x$ where $B = F(2,3,4,4,5,5)$ and $x$ is 7 where  
$\rank_2(T_1) =0$,  $\rank_4(T_4) =F_2 =1$, 
and $\rank_4(T_5) =F_3 =2$ where $T_i$ is the tiling in 
column $i$ for $i \in \{1,4,5\}$.  The rooks columns 2 and 6 
are in row 5 and the rook in column 3 is in row 3 so that $s_2 =s_6 =5$ 
and $s_3=3$. Thus 
\begin{eqnarray*}
\mathbf{W_{AugB_x,q}}(P) &=& -q^{0+(4+F_2)+(2+F_3)+1+2+(4+F_4)} = -q^{19}
\ \mbox{and} \\
\overline{\mathbf{W_{AugB_x,q}}}(P) &=& 
-q^{0+(4)+(2)+1+2+(4)} = -q^{13}
\end{eqnarray*}

\fig{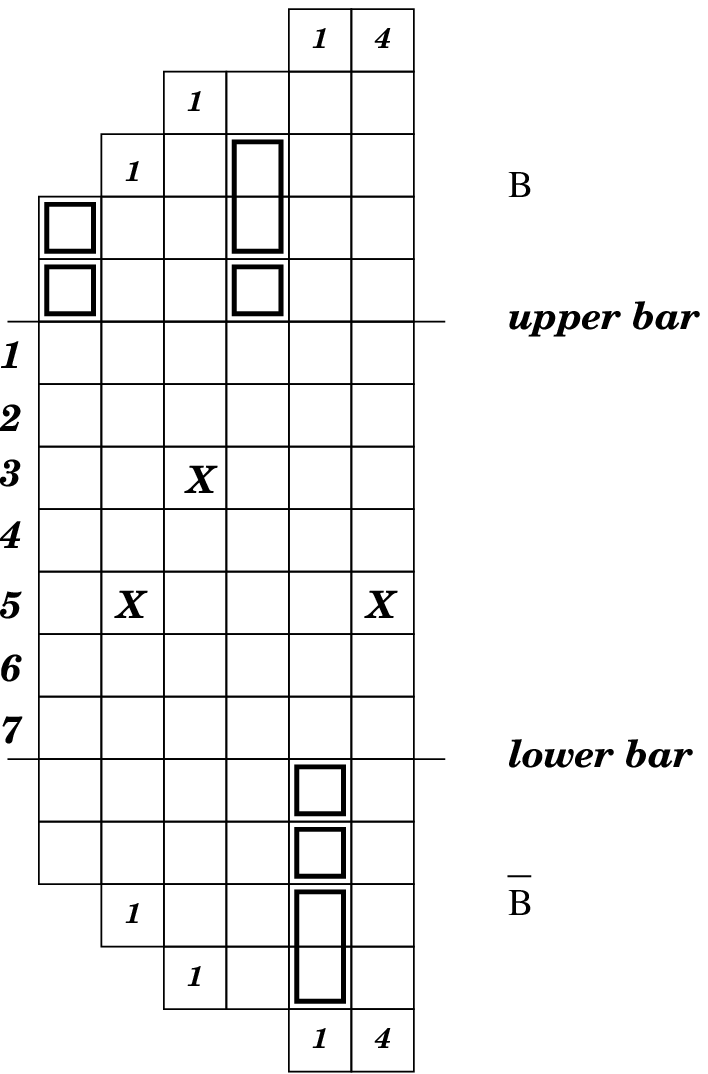}{A mixed rook placement.}

Our next theorem results from counting 
$\sum_{P \in \mathcal{M}_n(AugB_x)} \overline{\mathbf{W_{AugB_x,q}}}(P)$ 
in two different ways. 

\begin{theorem}\label{thm:ovqFerrersrookprod} 
Let $B =F(b_1, \ldots, b_n)$ be a Ferrers 
board where $0 \leq b_1 \leq \cdots \leq b_n$ and $b_n > 0$ 
and $x \in \Pos$. Then 
\begin{equation}\label{ovqFibrookprod}
\ [x]_q^n =
\sum_{k=0}^n \overline{\mathbf{RT}}_{n-k}(B,q) 
([x]_q-[F_{b_1}]_q) ([x]_q-[F_{b_2}]_q)\cdots 
([x]_q-[F_{b_k}]_q).
\end{equation}
\end{theorem}
\begin{proof}
Fix $x$ to be a positive integer and consider 
the sum $S=\sum_{P \in \mathcal{M}_n(AugB_x)} 
\overline{\mathbf{W_{AugB_x,q}}}(P)$. 
First we consider the contribution of each column as we proceed 
from left to right. Given our three choices 
in column 1, the contribution of our choice of the tilings of 
height $b_1$ in column 1 of $B$ is $[F_{b_1}]_q$, the choice 
of placing a rook in between the upper  bar and the lower is $[x]_q$, 
and the contribution of our choice of the tilings of 
height $b_1$ in column 1 of $\overline{B}$ is $-[F_{b_1}]_q$. 
Thus the contribution of our choices in 
column 1 to $S$ is $[F_{b_1}]_q+[x]_q -[F_{b_1}]_q = [x]_q$. 

In general, after we have processed our choices in 
the first $j$ columns, our cancellation scheme ensures 
that the number of uncanceled cells in $B$ and $\overline{B}$ in 
the $j$-th column is $b_i$ for some $i \leq j$. 
Thus given our three choices 
in column j, the contribution of our choice of the tilings of 
height $b_i$ in column $j$ of $B$ is $[F_{b_i}]_q$, the choice 
of placing a rook in between the upper  bar and the lower is $[x]_q$, 
and the contribution of our choice of the tilings of 
height $b_i$ in column $j$ of $\overline{B}$ is $-[F_{b_i}]_q$. 
Thus the contribution of our choices in 
column $j$ to $S$ is $[F_{b_i}]_q+[x]_q -[F_{b_i}]_q = [x]_q$.
It follows that $S = [x]_q^n$.

On the other 
hand, suppose that we fix a Fibonacci rook placement  
$P \in \mathcal{NT}_{n-k}(B)$. 
Then we want to compute the $S_P = \sum_{Q \in \mathcal{M}_n(AugB_x), 
Q \cap B = P} \overline{\mathbf{W_{AugB_x,q}}}(P)$ which is the sum 
of $\overline{\mathbf{W_{AugB_x,q}}}(P)$ over all 
mixed placements $Q$ such that $Q$ intersect $B$ equals $P$. 
Our cancellation scheme ensures that the number of uncanceled cells in $B$ and $\overline{B}$ 
in the $k$ columns that do not contain tilings in $P$ is 
$b_1, \ldots, b_k$ as we read from right to left. 
For each such $1 \leq i \leq k$, the factor that 
arises from either choosing a rook to be placed 
in between the upper bar and lower bar or a flipped 
Fibonacci tiling of height $b_i$ in $\overline{B}$ is 
$[x]_q-[F_{b_i}]_q$. It follows that 
$$S_P =\overline{\mathbf{W_{B,q}}}(P) \prod_{i=1}^k [x]_q-[F_{b_i}]_q.$$
Hence it follows that 
\begin{eqnarray*}
S &=&  \sum_{k=0}^n \sum_{P \in \mathcal{NT}_{n-k}(B)} S_P \\
&=& \sum_{k=0}^n \left(\prod_{i=1}^k [x]_q-[F_{b_i}]_q\right)  
\sum_{P \in \mathcal{NT}_k(B)} \overline{\mathbf{W_{B,q}}}(P) \\
&=& \sum_{k=0}^n \overline{\mathbf{RT}}_{n-k}(B,q) \left( \prod_{i=1}^k [x]_q-[F_{b_i}]_q\right).
\end{eqnarray*}
\end{proof}

\begin{theorem}\label{thm:qFerrersrookprod} 
Let $B =F(b_1, \ldots, b_n)$ be a Ferrers 
board where $0 \leq b_1 \leq \cdots \leq b_n$ and $b_n > 0$ 
and $x \geq b_n$. Then 
\begin{equation}\label{qFibrookprod}
\ [x]_q^n =
\sum_{k=0}^n \mathbf{RT}_{n-k}(B,q) [x-F_{b_1}]_q [x-F_{b_2}]_q\cdots 
[x-F_{b_k}]_q.
\end{equation}
\end{theorem}
\begin{proof}
It is easy to see from our cancellation scheme that 
$$\mathbf{RT}_{n-k}(B,q) =q^{F_{b_1}+ \cdots + F_{b_k}} 
\overline{\mathbf{RT}}_{n-k}(B,q).$$
Thus it follows from (\ref{ovqFibrookprod}) 
that 
$$
[x]_q^n =
\sum_{k=0}^n \mathbf{RT}_{n-k}(B,q)q^{-(F_{b_1}+ \cdots + F_{b_k})} 
([x]_q-[F_{b_1}]_q) ([x]_q-[F_{b_2}]_q)\cdots 
([x]_q-[F_{b_k}]_q).$$
However since $x \geq F_{b_i}$ for every $i$, 
$$[x]_q-[F_{b_i}]_q = q^{F_{b_i}}[x-F_{b_i}]_q$$
so that 
$$
[x]_q^n =
\sum_{k=0}^n \mathbf{RT}_{n-k}(B,q)  
[x-F_{b_1}]_q [x-F_{b_2}]_q\cdots 
[x-F_{b_k}]_q.$$
\end{proof}

Now consider the special case of the previous three  theorems 
when $B_n = F(0,1,2, \ldots, n-1)$. Then (\ref{qFibtilerec}) implies 
that 
$$\mathbf{RT}_{n+1-k}(B_{n+1},q) = q^{F_{k-1}}\mathbf{RT}_{n+1-k}(B_n,q) + 
[F_k]_q \mathbf{RT}_{n-k}(B_n,q).$$
Similarly (\ref{ovqFibtilerec}) implies 
that 
$$\overline{\mathbf{RT}}_{n+1-k}(B_{n+1},q) = 
\overline{\mathbf{RT}}_{n+1-k}(B_n,q) + 
[F_k]_q \overline{\mathbf{RT}}_{n-k}(B_n,q).$$

It then easily follows that for all $0 \leq k \leq n$, 
\begin{equation}\label{qSfnkpqinterp}
 \mathbf{SF}_{n,k}(q) = \mathbf{RT}_{n-k}(B_n,q) 
\end{equation}
and 
\begin{equation}\label{ovqSfnkpqinterp}
 \overline{\mathbf{SF}}_{n,k}(q) = \overline{\mathbf{RT}}_{n-k}(B_n,q).
\end{equation}

Note that $\mathbf{SF}_{n,0}(q) = \overline{\mathbf{SF}}_{n,0}(q) =0$ 
for all $n \geq 1$ since there are no Fibonacci rook placements in 
$\mathcal{NT}_n(B_n)$ since there are only $n-1$ non-zero columns. 
Moreover such a situation, we see that (\ref{qSfnkpqinterp}) 
implies that for $x \geq n$, 
\begin{equation*}
 [x]_q^n = 
\sum_{k=1}^n \mathbf{SF}_{n,k}(q) [x]_q [x-F_1]_q [x-F_2]_q\cdots 
[x-F_{k-1}]_q
\end{equation*}
Thus we have given a combinatorial proof of 
(\ref{SFdef}). 
Similarly, (\ref{ovqSfnkpqinterp}) 
implies that for $x \geq n$, 
\begin{equation*}
 [x]_q^n = 
\sum_{k=1}^n \overline{\mathbf{SF}}_{n,k}(q) [x]_q ([x]_q-[F_1]_q) 
([x]_q-[F_2]_q)\cdots 
([x]_q-[F_{k-1}]_q)
\end{equation*}
Thus we have given a combinatorial proof of 
(\ref{ovSFdef}).

\section{Identities for $\mathbf{SF}_{n,k}(q)$ and 
$\mathbf{cF}_{n,k}(q)$}

In this section, we shall derive various identities and special values 
for the Fibonacci analogues of the Stirling numbers $\mathbf{SF}_{n,k}(q)$, 
$\overline{\mathbf{SF}}_{n,k}(q)$, $\mathbf{cF}_{n,k}(q)$,  and 
$\overline{\mathbf{cF}}_{n,k}(q)$. 

Note that by (\ref{rel}), 
\begin{equation}\label{I1}
\mathbf{SF}_{n,k}(q) = q^{\sum_{i=1}^{k-1} F_i} \overline{\mathbf{SF}}_{n,k}(q).\end{equation}

Then we have the following theorem. 
\begin{theorem} \label{thm:id1}
\begin{enumerate}
\item $\overline{\mathbf{SF}}_{n,n}(q) =1$ and 
$\mathbf{SF}_{n,n}(q) =q^{\sum_{i=1}^{n-1} F_i}$.

\item $\overline{\mathbf{SF}}_{n,n-1}(q) =\sum_{i=1}^{n-1} [F_i]_q$ 
and 
$\mathbf{SF}_{n,n-1}(q) =q^{\sum_{i=1}^{n-2} F_i} \sum_{i=1}^{n-1} [F_i]_q$.
\item
$\overline{\mathbf{SF}}_{n,n-2}(q) =\sum_{i=1}^{n-2} [F_i]_q(\sum_{j=i}^{n-2} 
[F_j]_q)$
and 
$\mathbf{SF}_{n,n-2}(q) =q^{\sum_{i=1}^{n-3} F_i}  
\sum_{i=1}^{n-2} [F_i]_q(\sum_{j=i}^{n-2} 
[F_j]_q)$.

\item $\overline{\mathbf{SF}}_{n,1}(q) =1$ and 
$\mathbf{SF}_{n,1}(q) =1$.

\item $\overline{\mathbf{SF}}_{n,2}(q) =(n-1)$ and 
$\mathbf{SF}_{n,2}(q) =q(n-1)$.

\item $\overline{\mathbf{SF}}_{n,3}(q) =\frac{(1+q)^{n-1} -(q(n-1)+1)}{q^2}$ and 
$\mathbf{SF}_{n,3}(q) =(1+q)^{n-1} -(q(n-1)+1)$.
 
\end{enumerate}

\end{theorem}
\begin{proof}
For (1), it is easy to see that $\overline{\mathbf{SF}}_{n,n}(q) =1$ since 
the only placement in $\mathcal{FT}_{n-n}(B_n)$ is the empty placement. 
The fact that $\mathbf{SF}_{n,k}(q) =q^{\sum_{i=1}^{n-1} F_i}$ then 
follows from (\ref{I1}).

For (2), we can see that $\overline{\mathbf{SF}}_{n,n-1}(q) =
\sum_{i=1}^{n-1}[F_i]_q$ because placements in 
$\mathcal{FT}_{n-(n-1)}(B_n) = \mathcal{FT}_{1}(B_n)$ 
have exactly one column which 
is filled with a Fibonacci tiling.  If that column is 
column $i+1$, then $i \geq 1$ and 
the sum of the weights of the possible tilings 
in column $i$ is $[F_{i}]_q$.  The fact that 
$\mathbf{SF}_{n,n-1}(q) =q^{\sum_{i=1}^{n-2} F_i} \sum_{i=1}^{n-1} [F_i]_q$ then follows from (\ref{I1}).

For (3), we can classify the placements in 
$\mathcal{FT}_{n-(n-2)}(B_n) = \mathcal{FT}_{2}(B_n)$ by the left-most 
column which contains a tiling.  If that column is 
column $i+1$, then $i \geq 1$ and 
the sum of the weights of the possible tilings 
in column $i$ is $[F_i]_q$.  Moreover, any tiling 
in column $i$ cancels one cell in the remaining columns so 
that number of uncanceled cells in the columns to the right 
of column $i+1$ will be $i, \ldots, n-2$, reading from right to left. 
It then follows that 
$$\overline{\mathbf{SF}}_{n,n-2}(q) =\sum_{i=1}^{n-2} [F_i]_q(\sum_{j=i}^{n-2} 
[F_j]_q).$$ 
The fact that 
$$\mathbf{SF}_{n,n-1}(q) =q^{\sum_{i=1}^{n-3} F_i} \sum_{i=1}^{n-2} [F_i]_q(\sum_{j=i}^{n-2} 
[F_j]_q)$$ 
then follows from (\ref{I1}).

For (4), note that the elements in $\mathcal{FT}_{n-1}(B_n)$ have 
a tiling in every column. Given our cancellation scheme, there is 
exactly one such configuration. For example, the unique 
element of $\mathcal{FT}_{5}(B_6)$ is pictured in Figure 
\ref{fig:fullrooks} where we have placed $i$s in the cells 
canceled by the tiling in column $i$. Thus the unique 
element of $\mathcal{FT}_{n-1}(B_n)$ is just the Fibonacci 
rook placement where there is tiling of height one in each column. Thus 
$\overline{\mathbf{SF}}_{n,1}(q) =\mathbf{SF}_{n,1}(q) =1$ since 
the rank of each tiling height 1 is 0. 

\fig{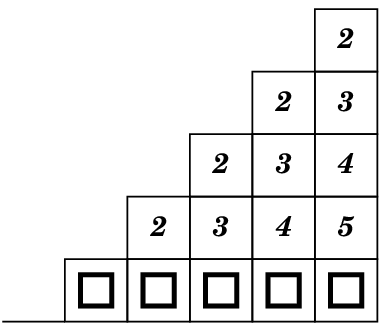}{The Fibonacci rook tiling in $\mathcal{FT}_{5}(B_6)$.}

For (5), note that the elements in $\mathcal{FT}_{n-2}(B_n)$ 
have exactly one column $i \geq 2$ which does not 
have a tiling.  Given our cancellation scheme, if the column 
with out a tiling is column $i \geq 2$, then any non-empty 
column to the left of column $i$ will be filled with a tiling 
of height $1$ and every column to the right of column $i$ will 
be filled with a tiling of height 2.  For example, the unique 
element of $\mathcal{FT}_{6}(B_8)$ is pictured in Figure 
\ref{fig:2fullrooks} where we have placed $i$s in the cells 
canceled by the tiling in column $i$. Since the ranks of 
the tilings of heights 1 and 2 are 0, it follows 
that $\overline{\mathbf{SF}}_{n,2}(q) =n-1$.  
The fact that 
$\mathbf{SF}_{n,n-1}(q) =q(n-1)$  
then follows from (\ref{I1}). 

\fig{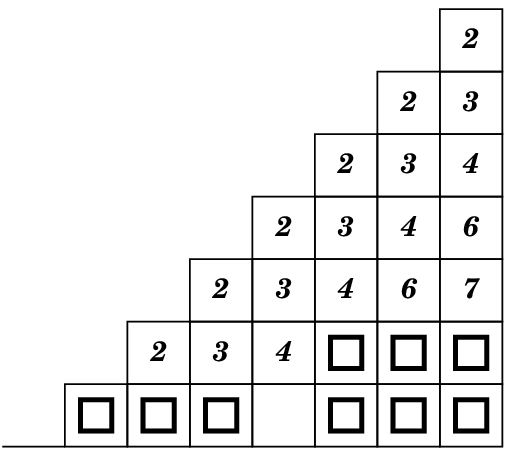}{A Fibonacci rook tiling in $\mathcal{FT}_{6}(B_8)$.}

For (6), we proceed by induction. Note that we have 
proved 
$$\mathbf{SF}_{3,3}(q) =q^{F_1+F_2}=q^2 = (1+q)^2-(2q +1).$$
Now assume that $n \geq 3$ and 
$\mathbf{SF}_{n,3}(q)=(1+q)^{n-1}-((n-1)q+1)$. Then 
\begin{eqnarray*}
\mathbf{SF}_{n+1,3}(q) &=&q^{F_2}\mathbf{SF}_{n,2}(q)+[F_3]_q
\mathbf{SF}_{n,3}(q)\\
&=& q\left(q(n-1) \right)+(1+q)\left((1+q)^{n-1}-((n-1)q+1)\right) \\
&=& q^2(n-1) +(1-q)^n -(n-1)q-(n-1)q^2 -q -1\\
&=& (1-q)^n -(nq+1).
\end{eqnarray*}
The fact that $\overline{\mathbf{SF}}_{n,3}(q)=
\frac{(1+q)^{n-1}-((n-1)q+1)}{q^2}$ then follows from (\ref{I1}). 
\end{proof}

Next we define  
$$\overline{\mathbb{SF}}_k(q,t) := 
\sum_{n \geq k} \overline{\mathbf{SF}}_{n,k}(q) t^n$$ 
for $k \geq 1$
It follows from Theorem \ref{thm:id1} that 
\begin{equation}\label{eq:id2}
\overline{\mathbb{SF}}_1(q,t) = 
\sum_{n \geq 1} \overline{\mathbf{SF}}_{n,1}(q) t^n = 
\sum_{n \geq 1} t^n = \frac{t}{1-t}.
\end{equation}
Then for $k > 1$,  
\begin{eqnarray*} 
\overline{\mathbb{SF}}_k(q,t) &=& \sum_{n \geq k} 
\overline{\mathbf{SF}}_{n,k}(q) t^n \\
&=& t^k + \sum_{n > k} \overline{\mathbf{SF}}_{n,k}(q) t^n \\
&=& t^k + t \sum_{n > k} \left( \overline{\mathbf{SF}}_{n-1,k-1}(q) +
[F_k]_q \overline{\mathbf{SF}}_{n-1,k-}(p,q)\right) t^{n-1} \\
&=& t^k + t\left(\sum_{n > k}  \overline{\mathbf{SF}}_{n-1,k-1}(q) t^{n-1}\right) 
+ [F_k]_qt\left(\sum_{n > k}  \overline{\mathbf{SF}}_{n-1,k}(q) t^{n-1}\right) \\
&=& t^k + t(\overline{\mathbb{SF}}_{k-1}(q,t) -t^{k-1}) +  [F_k]_qt 
\overline{\mathbb{SF}}_k(q,t).
\end{eqnarray*}
It follows 
that 
\begin{equation}\label{eq:id3}
\overline{\mathbb{SF}}_k(q,t) = \frac{t}{(1 - [F_k]_qt)}
\overline{\mathbb{SF}}_{k-1}(q,t).
\end{equation}
The following theorem easily follows from (\ref{eq:id2}) and 
(\ref{eq:id3}).
\begin{theorem}\label{thm:id4} For all $k \geq 1$, 
$$\overline{\mathbb{SF}}_k(q,t)= \frac{t^k}{(1-[F_1]_qt) (1-[F_2]_qt)\cdots 
(1-[F_k]qt)}.$$
\end{theorem}

Note that it follows from (\ref{I1}) and  Theorem \ref{thm:id4} that 
$$\mathbb{SF}_k(q,t)= \sum_{n \geq k}\mathbf{SF}_{n,k}(q) t^n = 
\frac{q^{\sum_{i=1}^{k-1}F_i} t^k}{(1-[F_1]_qt) (1-[F_2]_qt)\cdots 
(1-[F_k]qt)}.$$

For any formal power series in $f(x) = \sum_{n \geq 0}f_n x^n$, 
we let $f(x)|_{x^n} = f_n$ denote the coefficient of $x^n$ in 
$f(x)$. Our next result will give formulas for 
$\overline{\mathbf{SF}}_{n,k}(q)|_{q^s}$ for $s =0,1,2$. 

\begin{theorem}\label{thm:id5}
\begin{enumerate} 
\item For all $n \geq k \geq 1$, $\displaystyle 
\overline{\mathbf{SF}}_{n,k}(q)|_{q^0} = \binom{n-1}{k-1}$. 

\item For all $n > k \geq 2$, $\displaystyle 
\overline{\mathbf{SF}}_{n,k}(q)|_{q} = (k-2)\binom{n-1}{k}$.

\item  For all 
$n \geq s$, $\displaystyle 
\overline{\mathbf{SF}}_{n,3}(q)|_{q^s}=  \binom{n-1}{s+2}$.

\item  For all $n \geq k \geq 3$,  $\displaystyle 
\overline{\mathbf{SF}}_{n,k}(q)|_{q^2} = (k-3)\binom{n-1}{k} 
+ \binom{k-1}{2} \binom{n-1}{k+1}$. 

\item for all $n \geq k \geq 4$,
\begin{equation*}
\overline{\mathbf{SF}}_{n,k}(q)|_{q^3} = (k-4)\binom{n-1}{k} 
+ \left(\binom{k-1}{2} +\binom{k-2}{2} -1\right)\binom{n-1}{k+1} 
+\binom{k}{3} \binom{n-1}{k+2}.
\end{equation*}

\item For all 
$n \geq k \geq 4$,
\begin{eqnarray*}
\overline{\mathbf{SF}}_{n,k}(q)|_{q^4} &=& (k-4)\binom{n-1}{k} 
+ \left(\binom{k-1}{2} +\binom{k-2}{2} +\binom{k-3}{2}-3\right)\binom{n-1}{k+1} +\\
&& \left(2\binom{k}{3}+\binom{k-1}{3} -k +1\right) \binom{n-1}{k+2} + \binom{k+1}{4} \binom{n-1}{k+3}.
\end{eqnarray*}

\end{enumerate}
\end{theorem}
\begin{proof} 

For (1), note that a placement $P$ in $\mathcal{FT}_{n-k}(B_n)$ must 
have $k-1$ empty columns among columns $2, \ldots, n$.  If 
$\overline{WF}(P) =1$, then it must be the case that all the tilings  
in the columns which contain tilings in $P$ must have rank 0 so that 
the tiling must contain only tiles of height 1.  Thus $P$ is 
completely determined by the choice of the $k-1$ empty columns among 
columns $2, \ldots, n$.  Thus 
$\overline{\mathbf{SF}}_{n,k}(p,q)|_{q^0} = \binom{n-1}{k-1}$. 

For (3), note that by part 6 of Theorem \ref{thm:id1}, 
we have that for any $s \geq 0$, 
\begin{eqnarray*}
\overline{\mathbf{SF}}_{n,3}(q)|_{q^s} &=& \mathbf{SF}_{n,3}(q)|_{q^{s+2}} = (1+q)^{n-1}-((n-1)q+1)|_{q^{s+2}}\\
&=& \binom{n-1}{s+2}.
\end{eqnarray*}

For (2), note that 
$\overline{\mathbf{SF}}_{n,2}(q)|_{q} = 0$ since 
$\overline{\mathbf{SF}}_{n,k}(q) = (n-1)$ by part 5 of Theorem \ref{thm:id1}.
By (3), $\overline{\mathbf{SF}}_{n,3}(q)|_{q}= \binom{n-1}{3}$. 
Thus our formula holds for $n =2$ and $n=3$.

Next fix $k \geq 4$ and assume by induction that 
$\overline{\mathbf{SF}}_{n,k-1}(q)|_{q}= (k-3) \binom{n-1}{k-1}$ 
for all $n \geq k-1$.  Then we shall prove by induction on 
$n$ that $\overline{\mathbf{SF}}_{n,k}(q)|_{q}= (k-2) \binom{n-1}{k}$.
The base case $n =k$ holds since 
$\overline{\mathbf{SF}}_{k,}(q)=1$. But then assuming that 
$\overline{\mathbf{SF}}_{n,k}(q)|_{q}= (k-2) \binom{n-1}{k}$, 
we see that 
\begin{eqnarray*}
\overline{\mathbf{SF}}_{n+1,k}(q)|_{q} &=& 
\overline{\mathbf{SF}}_{n,k-1}(q)|_{q}+ 
((1+q+q^2+ \cdots + q^{F_k-1})\overline{\mathbf{SF}}_{n,k}(q))|_{q} \\
&=& (k-3) \binom{n-1}{k-1} + \overline{\mathbf{SF}}_{n,k}(q))|_{q^0} +
\overline{\mathbf{SF}}_{n,k}(q))|_{q}\\
&=& (k-3) \binom{n-1}{k-1} + \binom{n-1}{k-1} +(k-2) \binom{n-1}{k} \\
&=& (k-2) \binom{n}{k}.
\end{eqnarray*}

Parts (4), (5), and (6) can easily be proved by induction.  

For example, by (3),
$$ \overline{\mathbf{SF}}_{n,3}(q)|_{q^2} = \binom{n-1}{4}$$ 
so that our formula holds for $k=3$.  Now suppose 
that $k \geq 4$ and our formula holds for $k-1$. That is, 
$$ \overline{\mathbf{SF}}_{n,k-1}(q)|_{q^2} = (k-4) 
\binom{n-1}{k-1} + \binom{k-2}{2}\binom{n-1}{k}.$$
Next observe that $\overline{\mathbf{SF}}_{k,k}(q)|_{q^2} =0$ since 
 $\overline{\mathbf{SF}}_{k,k}(q) =1$ so that our formula 
holds for $n =k$.  Note also that for 
$k \geq 4$, $F_k \geq 3$. 
But then for $n \geq k \geq 4$,  
\begin{eqnarray*}
\overline{\mathbf{SF}}_{n+1,k}(q)|_{q^2} &=& 
\overline{\mathbf{SF}}_{n,k-1}(q)|_{q^2}+\left([F_k]_q 
\overline{\mathbf{SF}}_{n,k}(q)\right)|_{q^2} \\
&=& \overline{\mathbf{SF}}_{n,k-1}(q)|_{q^2}+\left((1+q+q^2)
\overline{\mathbf{SF}}_{n,k}(q)\right)|_{q^2} \\
&=& \overline{\mathbf{SF}}_{n,k-1}(q)|_{q^2}+
\overline{\mathbf{SF}}_{n,k}(q)|_{q^0}+
\overline{\mathbf{SF}}_{n,k}(q)|_{q}+\overline{\mathbf{SF}}_{n,k}(q)|_{q^2}\\
&=& (k-4) \binom{n-1}{k-1} + \binom{k-2}{2}\binom{n-1}{k} 
+ \binom{n-1}{k-1} + (k-2)\binom{n-1}{k}
\\
&&+ \overline{\mathbf{SF}}_{n,k}(q)|_{q^2} \\
&=& (k-3) \binom{n-1}{k-1} + \binom{k-1}{2}\binom{n-1}{k} + \overline{\mathbf{SF}}_{n,k}(q)|_{q^2}.
\end{eqnarray*}
This gives us a recursion for $\overline{\mathbf{SF}}_{n+1,k}(q)|_{q^2}$ in terms of $\overline{\mathbf{SF}}_{n,k}(q)|_{q^2}$ which we can iterate 
to prove that 
$$\overline{\mathbf{SF}}_{n,k}(q)|_{q^2} = 
(k-3) \binom{n-1}{k} + \binom{k-1}{2}\binom{n-1}{k+1}.$$

For (5), we first have to establish the base case $k=4$. 

\begin{eqnarray*}
\overline{\mathbf{SF}}_{n+1,4}(q)|_{q^3} &=& 
\overline{\mathbf{SF}}_{n,3}(q)|_{q^3}+\left([F_4]_q 
\overline{\mathbf{SF}}_{n,4}(q)\right)|_{q^3} \\
&=& \overline{\mathbf{SF}}_{n,3}(q)|_{q^3}+\left((1+q+q^2)
\overline{\mathbf{SF}}_{n,4}(q)\right)|_{q^3} \\
&=& \overline{\mathbf{SF}}_{n,k-1}(q)|_{q^3}+
\overline{\mathbf{SF}}_{n,k}(q)|_{q}+
\overline{\mathbf{SF}}_{n,k}(q)|_{q^2}+\overline{\mathbf{SF}}_{n,k}(q)|_{q^3}\\
&=& \binom{n-1}{5} + 2\binom{n-1}{4} +
\left(\binom{n-1}{4}+3\binom{n-1}{5}\right)
+\overline{\mathbf{SF}}_{n,k}(q)|_{q^3} \\
&=& 3\binom{n-1}{4} + 4\binom{n-1}{5} +\overline{\mathbf{SF}}_{n,4}(q)|_{q^3}.
\end{eqnarray*} 
This gives us a recursion for $\overline{\mathbf{SF}}_{n+1,4}(q)|_{q^3}$ in terms of $\overline{\mathbf{SF}}_{n,4}(q)|_{q^3}$ which we can iterate 
to prove that 
$$\overline{\mathbf{SF}}_{n,4}(q)|_{q^3} = 
3\binom{n-1}{5} + 4\binom{n-1}{6}.$$
Thus our formula for (5) holds for $k =4$.

Next assume that $k \geq 5$. First we note that
$ \overline{\mathbf{SF}}_{k,k}(q)|_{q^3} =0$ since 
 $\overline{\mathbf{SF}}_{k,k}(q) =1$ so that our formula 
holds for $n =k$.  Note also that for 
$k \geq 5$, $F_k \geq 5$. 
Now suppose our formula holds for $k-1$. That is, 
\begin{multline*}  
\overline{\mathbf{SF}}_{n,k-1}(q)|_{q^3} = \\
(k-5) 
\binom{n-1}{k-1} + \left(\binom{k-2}{2}+\binom{k-3}{2}-1\right)
\binom{n-1}{k} +\binom{k-1}{3}\binom{n-1}{k+1}.
\end{multline*}
Next observe that $\overline{\mathbf{SF}}_{k,k}(q)|_{q^3} =0$ since 
 $\overline{\mathbf{SF}}_{k,k}(q) =1$ so that our formula 
holds for $n =k$.  Note also that for 
$k \geq 4$, $F_k \geq 3$. 
But then for $n \geq k \geq 5$,  
\begin{eqnarray*}
\overline{\mathbf{SF}}_{n+1,k}(q)|_{q^3} &=& 
\overline{\mathbf{SF}}_{n,k-1}(q)|_{q^3}+\left([F_k]_q 
\overline{\mathbf{SF}}_{n,k}(q)\right)|_{q^3} \\
&=& \overline{\mathbf{SF}}_{n,k-1}(q)|_{q^2}+\left((1+q+q^2+q^3)
\overline{\mathbf{SF}}_{n,k}(q)\right)|_{q^3} \\
&=& \overline{\mathbf{SF}}_{n,k-1}(q)|_{q^3}+
\overline{\mathbf{SF}}_{n,k}(q)|_{q^0}+
\overline{\mathbf{SF}}_{n,k}(q)|_{q}+
\overline{\mathbf{SF}}_{n,k}(q)|_{q^2}+ 
\overline{\mathbf{SF}}_{n,k}(q)|_{q^3}\\
&=& (k-5) 
\binom{n-1}{k-1} + \left(\binom{k-2}{2}+\binom{k-3}{2}-1\right)
\binom{n-1}{k} +\binom{k-1}{3}\binom{n-1}{k+1}\\
&& + \binom{n-1}{k-1}+(k-2)\binom{n-1}{k}+ 
(k-3)\binom{n-1}{k} + \binom{k-1}{2}\binom{n-1}{k+1} \\
&&+ \overline{\mathbf{SF}}_{n,k}(q)|_{q^3}\\
&=& (k-4) \binom{n-1}{k-1} + \left(\binom{k-1}{2}+\binom{k-2}{2}-1\right)
\binom{n-1}{k} + \binom{k}{3}\binom{n-1}{k+1}
\\
&&+ \overline{\mathbf{SF}}_{n,k}(q)|_{q^3}. \\
\end{eqnarray*}
This gives us a recursion for $\overline{\mathbf{SF}}_{n+1,k}(q)|_{q^3}$ in terms of $\overline{\mathbf{SF}}_{n,k}(q)|_{q^3}$ which we can interate 
to prove that 
$$\overline{\mathbf{SF}}_{n,k}(q)|_{q^3} = 
(k-4) \binom{n-1}{k} + \left(\binom{k-1}{2}+\binom{k-2}{2}-1\right)
\binom{n-1}{k+1} + \binom{k}{3}\binom{n-1}{k+2}.$$

For (6), again, we first have to establish the base case $k=4$. 

\begin{eqnarray*}
\overline{\mathbf{SF}}_{n+1,4}(q)|_{q^4} &=& 
\overline{\mathbf{SF}}_{n,3}(q)|_{q^4}+\left([F_4]_q 
\overline{\mathbf{SF}}_{n,4}(q)\right)|_{q^4} \\
&=& \overline{\mathbf{SF}}_{n,3}(q)|_{q^4}+\left((1+q+q^2)
\overline{\mathbf{SF}}_{n,4}(q)\right)|_{q^4} \\
&=& \overline{\mathbf{SF}}_{n,k-1}(q)|_{q^4}+
\overline{\mathbf{SF}}_{n,k}(q)|_{q^2}+
\overline{\mathbf{SF}}_{n,k}(q)|_{q^3}+\overline{\mathbf{SF}}_{n,k}(q)|_{q^4}\\
&=& \binom{n-1}{6} + \binom{n-1}{4} +3\binom{n-1}{5} + 
3\binom{n-1}{5}+4\binom{n-1}{6} + \\
&&\overline{\mathbf{SF}}_{n,k}(q)|_{q^4} \\
&=& \binom{n-1}{4} + 6\binom{n-1}{5} + 5\binom{n-1}{6}+\overline{\mathbf{SF}}_{n,4}(q)|_{q^3}.
\end{eqnarray*} 
This gives us a recursion for $\overline{\mathbf{SF}}_{n+1,4}(q)|_{q^4}$ in terms of $\overline{\mathbf{SF}}_{n,4}(q)|_{q^4}$ which we can iterate 
to prove that 
$$\overline{\mathbf{SF}}_{n,4}(q)|_{q^4} =  \binom{n-1}{5} + 6\binom{n-1}{6} + 5\binom{n-1}{7}.$$
Thus our formula for (6) holds for $k =4$.

Next assume that $k \geq 5$. First we note that
$ \overline{\mathbf{SF}}_{k,k}(q)|_{q^4} =0$ since 
 $\overline{\mathbf{SF}}_{k,k}(q) =1$ so that our formula 
holds for $n =k$.  Note also that for 
$k \geq 5$, $F_k \geq 5$. 
Now suppose our formula holds for $k-1$. That is, 
\begin{eqnarray*}
\overline{\mathbf{SF}}_{n,k-1}(q)|_{q^4} &=& (k-5) 
\binom{n-1}{k-1} + \left(\binom{k-2}{2}+\binom{k-3}{2}+\binom{k-3}{2}-3\right)
\binom{n-1}{k} +\\
&&\left(\binom{k-1}{3}+\binom{k-2}{3}-(k-1)+1\right)\binom{n-1}{k+1} + \binom{k}{4}\binom{n-1}{k+2}.
\end{eqnarray*}
Next observe that $\overline{\mathbf{SF}}_{k,k}(q)|_{q^5} =0$ since 
 $\overline{\mathbf{SF}}_{k,k}(q) =1$ so that our formula 
holds for $n =k$.  Note also that for 
$k \geq 5$, $F_k \geq 5$. 
But then for $n \geq k \geq 5$,  
\begin{eqnarray*}
\overline{\mathbf{SF}}_{n+1,k}(q)|_{q^4} &=& 
\overline{\mathbf{SF}}_{n,k-1}(q)|_{q^4}+\left([F_k]_q 
\overline{\mathbf{SF}}_{n,k}(q)\right)|_{q^4} \\
&=& \overline{\mathbf{SF}}_{n,k-1}(q)|_{q^4}+\left((1+q+q^2+q^3+q^4)
\overline{\mathbf{SF}}_{n,k}(q)\right)|_{q^4} \\
&=& \overline{\mathbf{SF}}_{n,k-1}(q)|_{q^3}+
\overline{\mathbf{SF}}_{n,k}(q)|_{q^0}+
\overline{\mathbf{SF}}_{n,k}(q)|_{q}+
\overline{\mathbf{SF}}_{n,k}(q)|_{q^2}+ 
\overline{\mathbf{SF}}_{n,k}(q)|_{q^3}\\ 
&& + \overline{\mathbf{SF}}_{n,k}(q)|_{q^4}\\
&=& (k-5) 
\binom{n-1}{k-1} + \left(\binom{k-2}{2}+\binom{k-3}{2}+\binom{k-3}{2}-3\right)
\binom{n-1}{k} +\\
&&+ \left(\binom{k-1}{3}+\binom{k-2}{3}-(k-1)+1\right)\binom{n-1}{k+1} + 
\binom{k}{4}\binom{n-1}{k+2}\\
&&+ \binom{n-1}{k-1}+(k-2)\binom{n-1}{k}+ 
(k-3)\binom{n-1}{k} + \binom{k-1}{2} \binom{n-1}{k+1} \\
&&+ (k-4) \binom{n-1}{k} + \left(\binom{k-1}{2}+\binom{k-2}{2}-1\right)
\binom{n-1}{k+1} + \binom{k}{3} \binom{n-1}{k+2}\\
&&+ \overline{\mathbf{SF}}_{n,k}(q)|_{q^4}\\
&=& (k-4)\binom{n-1}{k-1} + 
\left( \binom{k-1}{2}+\binom{k-2}{2}+\binom{k-3}{2}-3 \right)
\binom{n-1}{k}\\
&&+\left(\binom{k}{3}+\binom{k-1}{3}-k+1\right) \binom{n-1}{k+1} +
\binom{k+1}{4}\binom{n-1}{k+2}  \\
&&+\overline{\mathbf{SF}}_{n,k}(q)|_{q^4}. 
\end{eqnarray*}
This gives us a recursion for $\overline{\mathbf{SF}}_{n+1,k}(q)|_{q^4}$ in terms of $\overline{\mathbf{SF}}_{n,k}(q)|_{q^4}$ which we can iterate 
to prove that 
\begin{eqnarray*}
\overline{\mathbf{SF}}_{n,k}(q)|_{q^4} &=& 
(k-4) \binom{n-1}{k} + \left(\binom{k-1}{2}+\binom{k-2}{2}+\binom{k-2}{2}-3\right)\binom{n-1}{k+1} + \\
&&\left(2\binom{k}{3}+\binom{k-1}{2}-k+1\right)
\binom{n-1}{k+2}+ \binom{k+1}{4}\binom{n-1}{k+3}.
\end{eqnarray*}
\end{proof}

A sequence of real numbers $a_0, \ldots, a_n$ is is said to be {\em unimodal} 
if there is a $0 \leq j \leq n$ such that 
$a_0 \leq  \cdots \leq a_j \geq a_{j+1} \geq \cdots \geq a_n$ and 
is said to be {\em log-concave} if for $0 \leq i \leq n$, 
$a_i^2- a_{i-1}a_{i+1} \geq 0$ where we set $a_{-1} = a_{n+1} =0$.
If a sequence is log-concave, then 
it is unimodal. A polynomial $P(x) = \sum_{k=0}^n a_k x^k$ is said to be {\em unimodal} if $a_0, \ldots, a_n$ is a unimodal sequence and is said 
to be log-concave if $a_0, \ldots, a_n$ is log concave. 

It is easy to see from Theorem \ref{thm:id1} that 
$\overline{\mathbf{SF}}_{n,k}(q)$ is unimodal for 
all $n \geq k$ when $k \in \{1,2,3\}$.  
Computational evidence suggests that 
$\overline{\mathbf{SF}}_{n,4}(q)$ is unimodal for all $n \geq 4$ and that 
$\overline{\mathbf{SF}}_{n,5}(q)$ is unimodal for all $n \geq 5$. 
However, it is not the case that  
$\overline{\mathbf{SF}}_{n,6}(q)$ is unimodal for all $n \geq 6$. 
For example, one can use part 3 of Theorem \ref{thm:id1} to compute 
\begin{multline*} 
\overline{\mathbf{SF}}_{8,6}(q) = \\
21 +28q +31q^2+29q^3+30q^4+25q^5+23q^2+22q^7+15q^8+10q^9+7q^{10}+5q^{11}+3q^{12}+2q^{13}+q^{14}.
\end{multline*}

It is not difficult to see that for any Ferrers board 
$B =F(b_1,\ldots,b_n)$, the coefficients that appear 
in the polynomials $\mathbf{FT}_{k}(B,q)$ and $\overline{\mathbf{FT}}_k(B,q)$ 
are essentially the same. That is, we have the following theorem. 

\begin{theorem}
Let $B=F(b_1, \ldots, b_n)$ be a skyline board. Then 
\begin{equation}\label{FTone}
\overline{\mathbf{FT}}_k(B,q) = \left( \prod_{i=1}^n (1+[F_{b_i}]_qz)\right)|_{z^k}
\end{equation}
and 
\begin{equation}\label{FTtwo}
\mathbf{FT}_k(B,q) = 
\left( \prod_{i=1}^n (q^{F_{b_i}}+[F_{b_i}]_qz)\right)|_{z^k} =
q^{\sum_{i=1}^n F_{b_i}}\left( \prod_{i=1}^n (1+\frac{1}{q}[F_{b_i}]_{\frac{1}{q}}z)\right)|_{z^k}
\end{equation}
\end{theorem}
\begin{proof}
It is easy to see that if we are creating a Fibonacci file tiling 
in $\mathcal{FT}_k(B)$, then in column $i$, we have two choices, 
namely, we can leave the column empty or put a Fibonacci tiling 
of height $b_i$.  For $\overline{\mathbf{FT}}_k(B,q)$, 
the weight of an empty column is 1 and the sum of weights of 
the Fibonacci tilings of height $b_i$ is $[F_{b_i}]_q$.  Thus 
$\left( \prod_{i=1}^n (1+[F_{b_i}]_qz)\right)|_{z^k}$ is equal 
to the sum over all Fibonacci file tilings where exactly $k$ columns 
have tiling which is equal to $\overline{\mathbf{FT}}_k(B,q)$.

Similarly, For $\mathbf{FT}_k(B,q)$, 
the weight of an empty column $i$ when it is empty is 
$q^{F_{b_i}}$ and the sum of weights of 
the Fibonacci tilings of height $b_i$ is $[F_{b_i}]_q$.  Thus 
$\left( \prod_{i=1}^n (q^{F_{b_i}}+[F_{b_i}]_qz)\right)|_{z^k}$ is equal 
to the sum over all Fibonacci file tilings where exactly $k$ columns 
have tiling which is equal to $\mathbf{FT}_k(B,q)$.
\end{proof}

It follow that for any $n$, the coefficient of $q^n$ in 
$\left( \prod_{i=1}^n (1+[F_{b_i}]_qz)\right)|_{z^k}$ is equal to 
the coefficient of $\frac{1}{q^{n+k}}$ in 
$\left( \prod_{i=1}^n (1+\frac{1}{q}[F_{b_i}]_{\frac{1}{q}}z)\right)|_{z^k}$.
It follows that 
$$\overline{\mathbf{FT}}_k(B,q)_{q^n} = 
\mathbf{FT}_k(B,q)|_{q^{-n-k+\sum_{i=1}^n F_{b_i}}}.$$

It is easy to see from (\ref{FTone}) that 
$$\overline{\mathbf{cF}}_{n,n-1}(q) = \sum_{i=1}^{n-1} [F_i]_q$$
so that coefficient of $q^k$ in $\overline{\mathbf{cF}}_{n,n-1}(q)$ weakly decreases as $k$ goes from 0 to $F_{n-1}-1$. It follows that the coefficient 
of $q^k$ in $\mathbf{cF})_{n,n-1}(q)$ weakly increase. Similarly, 
it is easy to see that 
$$\overline{\mathbf{cF}}_{n,1}(q) = \prod_{i=1}^{n-1}[F_i]_q$$
so that $\overline{\mathbf{cF}}_{n,1}(q)$ is just the rank generating 
function of a product of chains which is know to be symmetric and unimodal, 
see \cite{C}.

From our computational evidence, it seems that the polynomials 
$\overline{\mathbf{cF}}_{n,2}(q)$ 
are unimodal. However, it is not the case 
$\overline{\mathbf{cF}}_{n,k}(q)$ are unimodal for all $n$ and $k$.
For example, $\overline{\mathbf{cF}}_{9,7}(q)$ starts out 
$$28+42q+50q^2+53q^3+58q^4+57q^5+58q^6+60q^7+ \ldots .$$

Finally, our results show that the matrices  
$||(-1)^{n-k}\overline{\mathbf{cF}}_{n,k}(q)||$ and 
$||\overline{\mathbf{SF}}_{n,k}(q)||$ are inverses of each other. One can give a combinatorial proof of this fact.  Indeed, the combinatorial 
proof of \cite{BPR} which shows that matrices 
$||(-1)^{n-k}\mathbf{cf}_{n,k}(q)||$ and 
$||\mathbf{Sf}_{n,k}(q)||$ are inverses of each other can also be applied to show 
that the matrices  
$||(-1)^{n-k}\overline{\mathbf{cF}}_{n,k}(q)||$ and 
$||\overline{\mathbf{SF}}_{n,k}(q)||$ are inverses of each other.


\end{document}